\documentclass[autodetect-engine, 11pt]{article}
\usepackage{graphicx, color}
\usepackage{mymacro}
\usepackage{wasysym}
\usepackage[all]{xy}
\usepackage{typearea}
\renewcommand{\Sh}{\mathrm{Sh}}
\newcommand{\Zero}{\mathrm{Zero}}

\renewcommand{\Mod}{\mathrm{Mod}}
\newcommand{\forlim}{{\text{``${\displaystyle{\lim}}$''}}}

\newcommand{\musupp}{\mu\mathrm{supp}}
\newcommand{\shbarr}{{\footnotesize{\saturn}}}
\newcommand{\hbarr}{{\saturn}}
\newcommand{\potimes}{\overset{+}{\otimes}}
\newcommand{\Sol}{\mathrm{Sol}}

\title{Sheaf quantization from exact WKB analysis}
\author{Tatsuki Kuwagaki}
\date{}

\begin{document}

\maketitle
\begin{abstract}
A sheaf quantization is a sheaf associated to a Lagrangian brane. By using the results of exact WKB analysis, we sheaf-quantize spectral curves over the Novikov ring under some assumptions on the behavior of Stokes curves. 

For Schr\"odinger equations, we prove that the local system associated to the sheaf quantization (microlocalization a.k.a. abelianization) over the spectral curve can be identified with the Voros--Iwaki--Nakanishi coordinate.

We expect that these sheaf quantizations are the object-level realizations of the $\hbar$-enhanced Riemann--Hilbert correspondence.

\end{abstract}
\setcounter{tocdepth}{1}
\section{Introduction}
A quantization of a Lagrangian submanifold $L$ usually means a module supported on $L$ over a quantized symplectic manifold. For example, a $\cD$-module on a complex manifold is a quantization of its characteristic variety. Via the Riemann--Hilbert correspondence, one can introduce a notion of quantization with more topological flavor: {\em sheaf quantization}. A sheaf is a sheaf quantization of its microsupport $\SS$~\cite{KS}.

Characteristic varieties and microsupports are always conic. To quantize non-conic Lagrangians, it is classical to use $\cD^\hbar$-modules. A $\cD^\hbar$-module $\cE$ is a quantization of $\supp(\cE/\hbar\cE)$. Tamarkin introduced an analogue of this trick on the Riemann--Hilbert dual side~\cite{Tam}. Namely, for a real manifold $M$, consider the product $T^*M\times T^*_{\tau>0}\bR_t$ where $\bR_t$ is the real line with the standard coordinate $t$, $\tau$ is the dual cotangent coordinate of $t$, and $T^*_{\tau>0}\bR_t$ is the subset of $T^*\bR_t$ defined by $\tau>0$. Let $\rho$ be a map defined by $\rho\colon T^*M\times T^*_{\tau>0}\bR_t\rightarrow T^*M; (x, \xi, t,\tau)\mapsto (x, \xi/\tau)$ where $x\in M$ and $\xi\in T^*_xM$. Roughly speaking, we say that a sheaf $\cE$ on $M\times \bR_t$ is a {\em sheaf quantization} of $\musupp(\cE):=\rho(\SS(\cE)\cap (T^*M\times T^*_{\tau>0}\bR_t))$. This idea is quite useful to study symplectic topology. For example, Tamarkin proved some results on Hamiltonian non-displaceablity~\cite{Tam} and Guillermou proved some results on the nearby Lagrangian conjecture~\cite{Guillermou}. The statements obtained by sheaf quantizations are parallel to those obtained by Floer-theoretic methods. In fact, one of Tamarkin's first motivations to introduce $\bR_t$ is to realize the Novikov ring $\Lambda_0$ in sheaf theory~\cite[p.2, Remark]{Tam}. Also, Guillermou~\cite{guillermou2015quantization} proved that a sheaf quantization is associated to an exact Lagrangian brane, which is nothing but a necessary condition to be an object of Fukaya category.

In the first version of quantization (using $\cD$-modules), via the Riemann--Hilbert correspondence, a sheaf quantization over a complex manifold is equivalent to a quantization by a regular holonomic $\cD$-module. However, in the second version (using $\cD^\hbar$-modules), there is no such Riemann--Hilbert correspondence, so the relation is not clear.

One of the purpose of this paper is to reveal a part of the relationship explicitly, using exact WKB analysis. We will explore it more abstractly by constructing a functor relating $\cD^\hbar$-modules and sheaf quantizations in the subsequent publication~\cite{hRH}

Naively speaking, the Riemann--Hilbert functor is about solving differential equations. Similarly, in our case, we have to solve $\cD^\hbar$-modules. The appropriate method is {\em exact WKB analysis}~\cite{Voros, DDP, KawaiTakei}. Classically, the WKB approximation means a semi-classical approximation of quantum mechanics obtained by considering the Planck constant very small. Exact WKB analysis is a method to solve differential equations with $\hbar$ along this line. We solve a given differential equation with special formal power series (WKB solutions) in $\hbar$ and try to sum up the series by Borel's method. 

Exact WKB analysis is quite successful for the second order differential equations over complex 1-dimensional spaces. The result is, roughly speaking, for each $\hbarr\in \bC^\times_\hbar$, we have a chamber decomposition of the base space, solutions over each chamber, and connection formulas between solutions in different chambers. For us, these data give a gluing of sheaves.

\begin{theorem}
Let $C$ be a Riemann surface and $\cM$ be a rank 2 flat meromorphic $\hbar$-connection with Assumption~\ref{ourassumption}. For a sufficiently small and generic $\hbarr \in \bC^{\times}_\hbar$, there exists a sheaf quantization $S_{\cM}^{\shbarr}$ of the spectral curve of $\cM$ as an object of $\Sh_{\tau>0}^\bR(M\times \bR_t)$ and the $0$-th cohomology of its microlocalization is the Voros local system.
\end{theorem}
The category $\Sh^\bR_{\tau>0}(M\times \bR_t)$ is the $\bR$-equivariant version of the category $\Sh_{\tau>0}(M\times \bR_t)$ introduced by Tamarkin~\cite{Tam}. In the previous works~\cite{Guillermou, JT}, to sheaf-quantize a Lagrangian, we have to assume that the Lagrangian is exact. Since spectral curves are not exact, the previous framework cannot be applied to our situation. One of the technical advances of this paper is to use $\bR$-equivariant sheaves to treat nonexact Lagrangians. 
This introduction of equivariant sheaves naturally induces the Novikov ring action in sheaf theory as follows:
\begin{lemma}[Lemma~\ref{lem:Novcomputation}]
    The endomorphism ring of the monoidal unit of the equivariant version of Tamrkin's category is isomorphic to the Novikov ring.
\end{lemma}
This is the first direct observation of the Novikov ring in sheaf theory, although which has been anticipated since Tamarkin's work.
We will use this Novikov ring to study symplectic topology via sheaf theory in the subsequent work~\cite{IK}. It is also very interesting to have the Novikov ring in the context of exact WKB analysis, since a similar structure has already appeared in the theory under the name of {\em transseries}. A precise relationship between these two notions should be clarified in a future research.

As we have already mentioned, we expect our construction is a kind of Riemann--Hilbert correspondence, {\em $\hbar$-Riemann--Hilbert correspondence}. In the subsequent work~\cite{hRH}, we construct a functor from the category of $\cD^\hbar$-modules to the category of sheaf quantizations, although the coefficients we take is enlarged to obtain the Novikov rings on the side of $\cD^\hbar$-modules. For the second order case, our sheaf quantization here can be comparable to the sheaf quantization obtained via the $\hbar$-RH functor. For the higher order case, we expect that a similar coincidence is possible after exact WKB analysis is settled.

One can also ask relations to other realizations of irregular Riemann--Hilbert correspondence. We will discuss relationships to (1) D'Agnolo--Kashiwara's holonomic Riemann--Hilbert correspondence~\cite{DK}, and (2) Shende--Treumann--Williams--Zaslow's Legendrian knot description~\cite{STWZ}.

\begin{remark}
Microlocal sheaf theory has already met exact WKB analysis in the work of Getmanenko--Tamarkin \cite{GetmanekoTamarkin}. In contrast to our treatment, they applied microlocal sheaf theory to the problem after the Laplace transform. Hence their solution sheaf corresponds to Borel-transformed solutions. It may be interesting to know the relationship between our sheaves and their sheaves.
\end{remark}

\begin{remark}
Recently, Kontsevich--Soibelman~\cite{KonTalk, KontsevichSoibelman} announced a formulation of Riemann--Hilbert correspondence for deformation quantization. From this point of view, our construction (or our theorem in \cite{hRH}) can be considered as an intermediate step. The other step to prove Kontsevich--Soibelman's conjecture should be a version of Nadler--Zaslow equivalence~\cite{NZ}.
\end{remark}

\subsection*{Acknowledgment}
I would like to thank Kohei Iwaki, Tsukasa Ishibashi, Hiroshi Ohta, and Vivek Shende for many comments and discussions. Especially, Iwaki-san asked me to find relationship between exact WKB analysis and microlocal sheaf theory for some years. I am very happy to write this note as my ``homework report" of his lectures. I also would like to thank the anonymous referees for many comments.
This work was partially supported by World Premier International Research Center Initiative (WPI), MEXT, Japan
and JSPS KAKENHI Grant Number JP18K13405.

\subsection*{Notation}
\begin{enumerate}
    \item (Base field) Our base field is $\bK$. In a large part of this paper, $\bK$ is $\bC$.
    \item (Novikov ring) Let $R$ be a commutative ring. Consider $\bR_{\geq 0}$ as a semigroup and we denote the group ring by $R[\bR_{\geq0}]$. We denote the indeterminate of $R[\bR_{\geq 0}]$ by $T^a$ for $a\in \bR_{\geq 0}$. The Novikov ring $\Lambda^R_0$ is the projective limit \[\lim_{\substack{\longleftarrow \\a\rightarrow \infty}}R[\bR_{\geq 0}]/T^aR[\bR_{\geq 0}]\]. We denote the fraction field by $\Lambda^R$ and call it the Novikov field. If the context is clear, we omit the superscript $R$.
    \item (Category) A category usually means a dg-category unless specified. When dg-categories are involved, all the operations are understood to be derived.
    
    \item (Planck) We will use $\hbar$ for a general element or the coordinate function of $\bC^\times_\hbar$ and $\hbarr$ for some fixed element in $\bC_\hbar$.
    
    \item (Sheaf) Let $M$ be a topological space. For us, {\em sheaf} means {\em sheaf of $\bK$-vector spaces}.  For a locally closed subset $Z$ of $M$, we use $\bK_Z$ to be the constant sheaf of rank 1 supported on $Z$. If $Z$ is defined by some inequality $I$, we set $\bK_I:=\bK_Z$.
    
    \item (Complex cotangent bundle) For a complex manifold $M$, the cotangent bundle $T^*M$ is a holomorphic symplectic manifold by the standard symplectic structure. We denote the projection $T^*M\rightarrow M$ by $p_{T^*M}$.
     
    On a local coordinate $(U, \{z_i\})$, the standard symplectic form $\omega$ can be written as $\sum_i d\zeta_i\wedge dz_i$ where $\zeta_i$ is the cotangent coordinate dual to $z_i$. The form $\sum_i\zeta_i dz_i$ gives a global 1-form $\lambda$ on $T^*M$. We call $\lambda$ the Liouville form. 
     \item (Real cotangent bundle) Similarly, for a manifold $M$, the cotangent bundle $T^*M$ is a symplectic manifold by the standard exact symplectic structure. We denote the standard real symplectic form by $\omega_{st}$, and the standard Liouville form by $\lambda_{st}$.
\end{enumerate}

\section{Microlocal category}
Here we will introduce our category which is the living place of our sheaf quantizations. This category is a refined version of the category introduced by Tamarkin~\cite{Tam}. We will generalize the presentation of Guillermou--Schapira~\cite{GS}.

\subsection{Positively-microsupported sheaves}
Let $M$ be a connected differentiable manifold and $\bR_t$ be the real line with the standard coordinate $t$. We denote the derived category of $\bK$-module sheaves by $\Sh(M\times \bR_t)$. 

We denote the $i$-th projection by $\bR_t\times \bR_t\rightarrow \bR_t$. We also have the induced map $\id_M\times p_i\colon M\times \bR_t^2\rightarrow M\times \bR_t$ where $\id_M$ is the identity map of $M$. We also denote the addition map $\bR_t\times \bR_t\rightarrow \bR_t$ by $m_t$.

For objects $\cE, \cF\in \Sh(M\times \bR_t)$, we set
\begin{equation}
    \cE\star \cF:=(\id_M\times m_t)_!((\id_M\times p_1)^{-1}\cE\otimes (\id_M\times p_2)^{-1}\cF).
\end{equation}
We call it the convolution product.

We denote the cotangent bundle of $M$ by $T^*M$. Let $\bR_t$ be the real line with the standard coordinate $t$. The dual cotangent coordinate of $T^*\bR_t$ is denoted by $\tau$. The subset $\{\tau>0\}:=T^*_{\tau>0}(M\times \bR_t)=T^*M\times T^*_{\tau>0}\bR_t$ of $T^*M\times T^*\bR_t$ is defined by $\tau>0$.

We define the full subcategory $\Sh_{\tau\leq 0}(M\times \bR_t)$ of $\Sh(M\times \bR_t)$ as the full subcategory spanned by the objects satisfying $\SS(\cE)\subset \{\tau\leq 0\}$.

\begin{lemma}\label{lemma:nonequivvanishing}
For $\cE\in \Sh_{\tau\leq 0}(M\times \bR_t)$, we have $\cE\star \bK_{t\geq 0}\simeq 0$.
\end{lemma}
\begin{proof}
    This is a special case of \cite[Proposition 3.19 (b)]{GS}.
\end{proof}
We consider the skyscraper sheaf $\bK_{M\times \{t=0\}}$.
\begin{lemma}\label{lem:nonequivunit}
For $\cE\in \Sh_{\tau\leq 0}(M\times \bR_t)$, we have $\cE\star \bK_{M\times\{t=0\}}\cong \cE$. 
\end{lemma}
\begin{proof}
    This is a special case of \cite[p.25]{GS}.
\end{proof}
By the adjunction $\Hom(\bK_{t\geq 0}, \bK_{t=0})\cong \Hom(\bK_{t=0}, \bK_{t=0})$, we have the map $\bK_{t\geq 0}\rightarrow \bK_{t=0}$ corresponding to the identity.
By Lemma~\ref{lem:nonequivunit}, we have a morphism $\cE\star\bK_{t\geq 0}\rightarrow \cE$ for any $\cE\in \Sh(M\times \bR_t)$.
\begin{lemma}\label{lem:nonequivdecomposition}
The cone of the above morphism is in ${}^\perp\Sh_{\tau\leq 0}(M\times \bR_t)$. Moreover, the distinguished triangle gives a semi-orthogonal decomposition:
\begin{equation}
    \Sh(M\times \bR_t)\cong \la\Sh_{\tau\leq 0}(M\times \bR_t), {}^\perp\Sh_{\tau\leq 0}(M\times \bR_t)\ra.
\end{equation}
\end{lemma}
\begin{proof}
    This is a direct consequence of \cite[Proposition 3.21]{GS}.
\end{proof}
We set
\begin{equation}
    \Sh_{\tau>0}(M\times \bR_t):=\Sh(M\times \bR_t)/\Sh_{\tau\leq 0}(M\times \bR_t)
\end{equation}
where $/$ means the Drinfeld--Verdier dg-quotient.
We denote the quotient functor $\Sh(M\times \bR_t)\rightarrow \Sh_{\tau>0}(M\times \bR_t)$ by $q$.

By Lemma~\ref{lemma:nonequivvanishing}, the functor $(-)\star\bK_{t\geq 0}$ induces a functor $\iota_{>0}\colon \Sh_{\tau>0}(M\times \bR_t)\rightarrow \Sh(M\times \bR_t)$.
\begin{lemma}
    The functor $\iota_{>0}$ is fully faithful embedding, and gives an an equivalence onto ${}^\perp\Sh_{\tau\leq 0}(M\times \bR_t)$.
\end{lemma}
\begin{proof}
    This is a special case of \cite[Proposition 3.21]{GS}.
\end{proof}

\subsection{Equivariant sheaves}

For $c\in \bR$, we set
\begin{equation}
    T_c\colon M\times \bR_t\rightarrow M\times \bR_t; x\mapsto x+c.
\end{equation}
The isomorphisms $\{T_c\}_{c\in \bR}$ gives an action of $\bR$ on $M\times \bR_t$. Here we equip the group $\bR$ with the discrete topology.

\begin{definition}
An equivariant sheaf on $M\times \bR_t$ consists of the following data:
\begin{enumerate}
    \item a sheaf $\cE$ on $M\times \bR_t$,
    \item an isomorphism $\alpha_c\colon \cE\xrightarrow{\cong} T_c\cE:=T_{c*}\cE$ for any $c\in \bR$ 
\end{enumerate}
such that $T_{c*}(\alpha_{c'})=\alpha_{c+c'}$ for any $c, c '\in \bR$. A morphism between equivariant sheaves is a morphism of sheaves compatible with $\alpha_c$'s.
\end{definition}

Let $\Sh^\bR_{\heartsuit}(M\times \bR_t)$ be the abelian category of $\bR$-equivariant sheaves and $\Sh^\bR(M\times \bR_t)$ be the derived dg-category of $\Sh^\bR_\heartsuit(M\times \bR_t)$. Note that the equivariant derived category in the sense of Bernstein--Lunts~\cite{BernsteinLunts} coincide with the naive derived category in this case, since $\bR$ is a discrete group.

We have the forgetful functor
\begin{equation}
    \frakf\colon \Sh^\bR(M\times \bR_t)\rightarrow \Sh(M\times \bR_t).
\end{equation}

\begin{definition}
    A functor $F\colon \cA\rightarrow \cB$ between additive categories is conservative if $F(a)\simeq 0$ implies that $a\simeq 0$.
\end{definition}

\begin{lemma}
    Let $F\colon \cA\rightarrow \cB$ be an exact conservative functor between triangulated categories. Then $f\colon \cE\rightarrow \cF$ in $\cA$ is an isomorphism if $F(f)$ is an isomorphism.
\end{lemma}
\begin{proof}
    Suppose $F(f)$ is an isomorphism. Then $F(\Cone(f))=\Cone(F(f))\simeq 0$. By the conservativity, $\Cone(f)\simeq 0$. This completes the proof.
\end{proof}

\begin{lemma}\label{lemma:conservative}
The forgetful functor $\frakf$ is conservative.
\end{lemma}
\begin{proof}
For $\cE\in \Sh^\bR(M\times \bR_t)$, if $\frakf(\cE)\simeq 0$, then $\frakf(\cE)\rightarrow 0$ is a quasi-isomorphism. Since 
 this morphism can be lifted to $\cE\rightarrow 0$, we obtain the desired result.
\end{proof}

For an object $\cE\in \Sh(M\times \bR_t)$, the direct sum $\bigoplus_{c\in \bR} T_c\cE$ has an obvious equivariant structure
\begin{equation}
    \alpha_{c'}\colon \bigoplus_{c\in \bR} T_c\cE\xrightarrow{\id} \bigoplus_{c\in \bR} T_c\cE=T_{c'}\bigoplus_{c\in \bR} T_c\cE.
\end{equation}
The assignment
\begin{equation}
    \cE\mapsto  \frakf^L(\cE):=\bigoplus_{c\in \bR} T_c\cE
\end{equation}
defines a functor $\frakf^L\colon \Sh(M\times \bR_t)\rightarrow \Sh^\bR(M\times \bR_t)$.
\begin{proposition}
    The functor $\frakf^L$ is the left adjoint of $\frakf$.
\end{proposition}
\begin{proof}
    Since the both functors are induced from exact functors between abelian categories, it is enough to show that the adjoint holds on the abelian level. We denote the abelian version of the functors by
    \begin{equation}
        \frakf_\heartsuit\colon \Sh^\bR_{\heartsuit}(M\times \bR_t)\substack{\leftarrow\\ \rightarrow}\Sh_\heartsuit(M\times \bR_t)\colon \frakf^L_\heartsuit.
    \end{equation}
    For objects $\cE\in \Sh^\bR_{\heartsuit}(M\times \bR_t)$ and $\cF\in \Sh_{\heartsuit}(M\times \bR_t)$, we have
    \begin{equation}
        \begin{split}
            \Hom_{\Sh^\bR_{\heartsuit}(M\times \bR_t)}(\frakf^L_\heartsuit\cF, \cE)&\cong \Hom_{\Sh^\bR_{\heartsuit}(M\times \bR_t)}(\bigoplus_{c\in \bR}T_c\cF, \cE).
        \end{split}
    \end{equation}
We would like to compute the right hand side. First, we have a morphism
\begin{equation}
    \Hom_{\Sh^\bR_{\heartsuit}(M\times \bR_t)}(\bigoplus_{c\in \bR}T_c\cF, \cE)\rightarrow \Hom_{\Sh_{\heartsuit}(M\times \bR_t)}(\bigoplus_{c\in \bR}T_c\cF, \frakf(\cE))
\end{equation}
induced by $\frakf$. Next, we consider the pull-back along the inclusion $\cF=T_0\cF\hookrightarrow \bigoplus_{c\in \bR}T_c\cF$:
\begin{equation}
    \Hom_{\Sh_{\heartsuit}(M\times \bR_t)}(\bigoplus_{c\in \bR}T_c\cF, \frakf(\cE))\rightarrow \Hom_{\Sh_{\heartsuit}(M\times \bR_t)}(\cF, \frakf(\cE)).
\end{equation}
Composing these two morphisms, we obtain a morphism
\begin{equation}
    C\colon \Hom_{\Sh^\bR_{\heartsuit}(M\times \bR_t)}(\bigoplus_{c\in \bR}T_c\cF, \cE)\rightarrow \Hom_{\Sh_{\heartsuit}(M\times \bR_t)}(\cF, \frakf(\cE)).
\end{equation}
To complete the proof, it is enough to check that this is an isomorphism. For any element $f\in \Hom_{\Sh_{\heartsuit}(M\times \bR_t)}(\cF, \frakf(\cE))$, we set 
\begin{equation}
    f_c:=\alpha_c^{-1}\circ T_{c*}(f)\colon T_c\cF\xrightarrow{} T_c\cE\xrightarrow{\alpha_c^{-1}}\cE.
\end{equation}
Then $\prod_cf_c\colon \bigoplus_{c\in \bR} T_c\cF\rightarrow \cE$ defines an element of $\Hom_{\Sh^\bR_{\heartsuit}(M\times \bR_t)}(\bigoplus_{c\in \bR}T_c\cF, \cE)$ such that $C(\prod_cf_c)=f$. This construction gives the inverse of $C$. This completes the proof.
\end{proof}

For an object $\cE\in \Sh(M\times \bR_t)$, the microsupport $\SS(\cE)$ is defined by \cite{KS}. Although, \cite{KS} defined it for the bounded case, it is well-known that the definition works well for the unbounded case.

\begin{definition}[Microsupport]
For an object $\cE\in \Sh^\bR(M\times \bR_t)$, we set
\begin{equation}\label{eq:definitionofequivariantSS}
    \SS(\cE):=\SS(\frakf(\cE)).
\end{equation}
\end{definition}

\subsection{Operations for equivariant sheaves}
We first recall some basic operations for equivariant sheaves. For details, we refer to \cite{BernsteinLunts, Grothendieck}.

Let $G$ be a group. Let $X_1, X_2$ be $G$-spaces. Let $f\colon X_1\rightarrow X_2$ be a $G$-map. Then we have functors $f^G_*, f^G_!\colon \Sh^G(X_1)\rightarrow \Sh^G(X_2),
    f_G^{-1}, f_G^! \colon \Sh^G(X_2)\rightarrow \Sh^G(X_1)$.
Standard adjunctions hold for these functors.

Let $G_1, G_2$ be groups. Let $X_i$ be $G_i$-spaces for $i=1,2$. We then have the tensor product functor
\begin{equation}
    \boxtimes^{G_1\times G_2}\colon \Sh^{G_1}(X_1)\times \Sh^{G_2}(X_2)\rightarrow \Sh^{G_1\times G_2}(X_1\times X_2).
\end{equation}

We consider $M\times \bR_t^2$ on which $\bR^2$ acts by the addition on each component.
Through the addition $m\colon \bR^2\rightarrow \bR$, the group $\bR^2$ also acts on $M\times \bR_t$. The kernel of the map $m\colon\bR^2\rightarrow \bR$ is the anti-diagonal $\Delta_a:=\{(a, -a)\in \bR\times \bR\}$.

We consider the addition map $\id_M\times m_t\colon M\times \bR_t\times \bR_t\rightarrow M\times \bR_t$ on $\bR_t$-factors. We then have a functor
\begin{equation}
    (\id_M\times m_t)^{\bR^2}_!\colon \Sh^{\bR^2}(M\times \bR_t^2)\rightarrow\Sh^{\bR^2}(M\times \bR_t).
\end{equation}

Suppose $\phi\colon G\rightarrow H$ be a surjective group homomorphism. Let $Y$ be an $H$-space. Then $G$ acts on $Y$ through $\phi$. Let $K$ be the kernel of $G\rightarrow H$. Then we have the invariant functor
\begin{equation}
    (-)^K\colon \Sh^G(Y)\rightarrow \Sh^H(Y)
\end{equation}
and the coinvariant functor
\begin{equation}
    (-)_K\colon \Sh^G(Y)\rightarrow \Sh^H(Y).
\end{equation}

Also, if one has an $H$-equivariant sheaf on $Y$, it can also be considered as a $G$-equivariant sheaf:
\begin{equation}
    \iota^\phi\colon \Sh^H(Y)\rightarrow \Sh^G(Y).
\end{equation}
The following is a standard fact.
\begin{lemma}
$\iota^\phi$ is the right adjoint of $(-)_K$ and the left adjoint of $(-)^K$.
\end{lemma}

In these notations, we have a functor
\begin{equation}
    (-)_{\Delta_a}\colon \Sh^{\bR^2}(M\times \bR_t)\rightarrow \Sh^{\bR}(M\times \bR_t).
\end{equation}

We set 
\begin{equation}
    (m_!^{\bR^2})_{\Delta_a}:=(-)_{\Delta_a}\circ m^{\bR^2}_!\colon \Sh^{\bR^2}(M\times \bR_t^2)\rightarrow\Sh^{\bR}(M\times \bR_t).
\end{equation}
We have the right adjoint of $(m_!^{\bR^2})_{\Delta_a}$:
\begin{equation}
(m_!^{\bR^2})_{\Delta_a}^R:=m^!_{\bR^2}\circ \iota^{a}\colon\Sh^{\bR}(M\times \bR_t)\rightarrow\Sh^{\bR^2}(M\times \bR_t^2).
\end{equation}

Let $p_i\colon M\times \bR_t^2\rightarrow M\times \bR_t$ be the $i$-th projection. We also have the corresponding projection $q_i\colon \bR^2\rightarrow \bR$. We then have
\begin{equation}
   (p_{i*}^{\bR^2})^{\ker q_i}:=(-)^{\ker q_i}\circ p^{\bR^2}_{i*}\colon  \Sh^{\bR^2}(M\times \bR^2_t)\rightarrow \Sh^{\bR}(M\times \bR_t)
\end{equation}
and 
\begin{equation}
   {}^L(p_{i*}^{\bR^2})^{\ker q_i}:=p_{i, \bR^2}^{-1}\circ \iota^{q_i} \colon  \Sh^{\bR}(M\times \bR_t)\rightarrow \Sh^{\bR^2}(M\times \bR^2_t).
\end{equation}

\begin{lemma}
${}^L(p_{i*}^{\bR^2})^{\ker q_i}$ is the left adjoint of $p_{i*}^{\ker q_i}$.
\end{lemma}

\begin{definition}[Convolution product]
We set, for objects $\cE, \cF\in \Sh^{\bR}(M\times \bR_t)$, 
\begin{equation}
    \cE \star_\bR \cF:=(m_!^{\bR^2})_{\Delta_a}((\id_M\times p_1)^{-1}_{\bR}\cE\otimes^{\bR\times \bR} (\id_M\times p_2)_\bR^{-1}\cF).
\end{equation}
\end{definition}

\begin{definition}[Convolution product]\label{def:convolution1}
    We set, for objects $\cE\in \Sh^\bR(M\times \bR_t)$ and $\cF\in \Sh(M\times \bR_t)$, 
    \begin{equation}
    \cE \star_{R} \cF:=\cE\star_\bR \frakf^L(\cF)\in \Sh^{\bR}(M\times \bR_t).
\end{equation}
\end{definition}

\begin{lemma}\label{lem:associativity}
    \begin{enumerate}
        \item For $\cE, \cF, \cG\in \Sh^{\bR}(M\times \bR_t)$, we have $(\cE\star_\bR\cF)\star_\bR\cG\cong \cE\star_\bR (\cF\star_\bR \cG)$.
        \item For $\cE, \cF\in \Sh^{\bR}(M\times \bR_t)$ and $\cG\in \Sh(M\times \bR_t)$, we have $(\cE\star_\bR\cF)\star_R\cG\cong \cE\star_\bR (\cF\star_R \cG)$.
        \item For $\cE\in \Sh^{\bR}(M\times \bR_t)$ and $\cF, \cG\in \Sh(M\times \bR_t)$, we have $(\cE\star_R\cF)\star_R\cG\cong \cE\star_R (\cF\star \cG)$.
    \end{enumerate}
\end{lemma}
\begin{proof}
    The proofs are straightforward. We omit the proof.
\end{proof}

For $\star_R$, the underlying sheaf is obtained as the nonequivariant convolution product:
\begin{lemma}\label{lemma:forgetconvolution}
    For objects $\cE\in \Sh^\bR(M\times \bR_t)$ and $\cF\in \Sh(M\times \bR_t)$, we have
    \begin{equation}
        \frakf(\cE\star_R\cF)\cong \frakf(\cE)\star \cF.
    \end{equation}
\end{lemma}
\begin{proof}
We have
\begin{equation}
    \frakf(\cE)\star \frakf^L(\cF)\cong \bigoplus_{c\in \bR}T_c(\frakf(\cE)\star \cF)\cong \bigoplus_{c\in \bR}(T_c\frakf(\cE))\star \cF\cong \bigoplus_{c\in \bR}\frakf(\cE)\star \cF
\end{equation}
where the last isomorphism is coming from the equivariant structure morphism of $\frak\cE$. Under this identification, the pull back along $a\in \Delta_a$ is $\bigoplus_{c-a\in \bR}\frakf(\cE)\star \cF$. By taking the coinvariant, we obtain $\frakf(\cE)\star \cF$. This completes the proof.
\end{proof}
Note that the equivariant structure of $\cE\star \cF$ is given by $\frakf(\cE)\star \cF\xrightarrow{\alpha_c\star\id}(T_c\frakf(\cE))\star \cF$ by the proof of the above lemma.

\begin{lemma}\label{lem:unit}
    We have
    \begin{equation}
        (-)\star_R\bK_{M\times \{t=0\}}\cong \id.
    \end{equation}
\end{lemma}
\begin{proof}
By Lemma~\ref{lemma:forgetconvolution} and Lemma~\ref{lem:nonequivunit}, we have 
\begin{equation}
    \frakf(\cE\star_R\bK_{M\times \{t=0\}})\cong \frakf(\cE)\star \bK_{M\times \{t=0\}}\cong \frakf(\cE).
\end{equation}
By the above remark, the equivariant structures also coincide.
\end{proof}

\subsection{Positively microsupported equivariant sheaves}
In this section, we would like to consider the equivariant version of positively microsupported sheaves.

We define the full subcategory $\Sh_{\tau\leq 0}^\bR(M\times \bR_t)$ of $\Sh^\bR(M\times \bR_t)$ as the full subcategory spanned by the objects satisfying $\SS(\cE)\subset \{\tau\leq 0\}$.

\begin{lemma}\label{lem:inducedf}
    The functor $\frakf$ can be restricted to $\frakf\colon \Sh^{\bR}_{\tau\leq 0}(M\times \bR_t)\rightarrow \Sh_{\tau\leq 0}(M\times \bR_t)$.
\end{lemma}
\begin{proof}
This follows from the definition of $\SS$ for equivariant sheaves~(\ref{eq:definitionofequivariantSS}).    
\end{proof}

Recall the notation $\bK_{t\geq 0}:=\bK_{M\times \{t\geq 0\}}$. 
\begin{lemma}\label{lemma:vanishing}
For $\cE\in \Sh^\bR_{\tau\leq 0}(M\times \bR_t)$, we have $\cE\star_R \bK_{t\geq 0}\simeq 0$.
\end{lemma}
\begin{proof}
 For $\cE\in \Sh^\bR_{\tau\leq 0}(M\times \bR_t)$, we have
    \begin{equation}
        \frakf(\cE\star_R \bK_{t\geq 0})\cong \frakf(\cE)\star \bK_{t\geq 0}
    \end{equation}
    by Lemma~\ref{lemma:forgetconvolution}. Combining Lemma~\ref{lem:inducedf} and Lemma~\ref{lemma:nonequivvanishing}, we can conclude that the right hand side is zero. By Lemma~\ref{lemma:conservative}, this implies that $\cE\star_R\bK_{t\geq 0}$ is also zero. This completes the proof.
\end{proof}

By Lemma~\ref{lem:unit}, $(-)\star_R\bK_{M\times \{t=0\}}\simeq \id$, we have a canonical morphism
\begin{equation}
    (-)\star_R\bK_{t\geq 0}\rightarrow \id
\end{equation}
induced by the canonical morphism $\bK_{t\geq 0}\rightarrow \bK_{M\times \{t= 0\}}$ corresponding to $\id$ under the adjunction isomorphism $\Hom(\bK_{t\geq 0}, \bK_{M\times\{t= 0\}})\cong \Hom(\bK_{M\times \{t= 0\}}, \bK_{M\times \{t=0\}})$.

In particular, we have a morphism
\begin{equation}
    ((-)\star_R\bK_{t\geq 0})\star_R\bK_{t\geq 0}\rightarrow(-)\star_R \bK_{t\geq 0}.
\end{equation}
\begin{lemma}\label{lem:idem}
The above morphism is an isomorphism
\end{lemma}
\begin{proof}
By Lemma~\ref{lem:associativity}, we have
\begin{equation}
    ((-)\star_R\bK_{t\geq 0})\star_R\bK_{t\geq 0}\cong (-)\star_R(\bK_{t\geq 0}\star\bK_{t\geq 0})\cong (-)\star\bK_{t\geq 0}. 
\end{equation}
The latter isomorphism is well-known.
\end{proof}

We denote the essential image of $(-)\star_R\bK_{t\geq 0}$ by $\Im(\bK_{t\geq 0})$.
\begin{lemma}\label{lem:imageistriangulated}
The full subcategory $\Im(\bK_{t\geq 0})$ is closed under taking cones.
\end{lemma}
\begin{proof}
    For any morphism $f\colon \cE\rightarrow \cF$ in $\Im(\bK_{t\geq 0})$, the natural transformation $(-)\star_R \bK_{t\geq 0}\rightarrow \id$ gives a commutative square
    \begin{equation}
        \xymatrix{
\cE \ar[d] \ar[r]^f &\cF \ar[d]  \\
\cE\star_R \bK_{t\geq 0}\ar[r]^{f\star_R \bK_{t\geq 0}} & \cF\star_R \bK_{t\geq 0}.
}
    \end{equation}
Since the vertical arrows are isomorphisms by Lemma~\ref{lem:idem}, we have
\begin{equation}
    \Cone(f)\cong \Cone(f\star_R\bK_{t\geq 0})\cong \Cone(f)\star_R\bK_{t\geq 0}.
\end{equation}
This completes the proof.
\end{proof}

\begin{lemma}\label{lem:decomposition}
    For any object $\cE\in \Sh^\bR(M\times \bR_t)$, we have a distinguished triangle
    \begin{equation}
        \cE_1\rightarrow \cE\rightarrow \cE_2\xrightarrow{[1]}
    \end{equation}
    such that $\cE_1\in \Im(\bK_{t\geq 0})$ and $\cE_2\in \Sh_{\tau\leq 0}^\bR(M\times \bR_t)$.
\end{lemma}
\begin{proof}
We set $\cE_1:=\cE\star_R\bK_{t\geq 0}$. By Lemma~\ref{lemma:forgetconvolution},
\begin{equation}
    \frakf(\cE_2)\cong \Cone(\frakf(\cE_1)\rightarrow \frakf(\cE))\cong \Cone(\frakf(\cE)\star \bK_{t\geq 0}\rightarrow \frakf(\cE)).
\end{equation}
The microsupport of the RHS is contained in $\{\tau\leq 0\}$ by Lemma~\ref{lem:nonequivdecomposition}. Then we have $\cE_2\in \Sh_{\tau\leq 0}^\bR(M\times \bR_t)$. This completes the proof.
\end{proof}

\begin{lemma}\label{lem:SOD}
    We have a semi-orthogonal decomposition
    \begin{equation}
        \Sh^\bR(M\times \bR_t)\cong \la \Sh^\bR_{\tau\leq 0}(M\times \bR_t), {}^\perp\Sh^\bR_{\tau\leq 0}(M\times \bR_t)\ra.
    \end{equation}
    Moreover, $\Im(\star_R\bK_{t\geq 0})$ is contained in ${}^\perp\Sh^\bR_{\tau\leq 0}(M\times \bR_t)$, and the inclusion is an equivalence.
\end{lemma}
\begin{proof}
    By Lemma~\ref{lem:idem} and Lemma~\ref{lem:decomposition}, $\Im(\bK_{t\geq 0})$ is an admissible subcategory and get a semi-orthogonal decomposition 
        \begin{equation}
        \Sh^\bR(M\times \bR_t)\cong \la \Sh^\bR_{\tau\leq 0}(M\times \bR_t), \Im(\star_R\bK_{t\geq 0})\ra.
    \end{equation}
    This completes the proof.
\end{proof}

We set
\begin{equation}
    \begin{split}
       \Sh^\bR_{\tau>0}(M\times \bR_t)&:=\Sh^\bR(M\times \bR_t)/\Sh^\bR_{\tau\leq 0}(M\times \bR_t).
    \end{split}
\end{equation}
To simplify the notations, we set
\begin{equation}
    \mu(T^*M):=\Sh^\bR_{\tau>0}(M\times \bR_t).
\end{equation}

We denote the quotient functor by
\begin{equation}
    \begin{split}
        q_\bR&\colon \Sh^\bR(M\times \bR_t)\rightarrow \mu(T^*M).
    \end{split}
\end{equation}
\begin{lemma}\label{lem:quotientequiv}
The quotient functor restricts to an equivalence $q_\bR|_{{}^\perp\Sh^\bR_{\tau\leq 0}(M\times \bR_t)}\colon {}^\perp\Sh^\bR_{\tau\leq 0}(M\times \bR_t)\rightarrow \mu(T^*M)$.
\end{lemma}
\begin{proof}
    This is straightforward from Lemma~\ref{lem:SOD}.
\end{proof}

Lemma~\ref{lemma:vanishing} implies that the functor $(-)\star_R \bK_{t\geq 0}$ descends to a functor from $\mu(T^*M)$ to $\Sh^\bR(M\times \bR_t)$. We denote the induced functor by $\iota^\bR_{>0}$.

\begin{lemma}\label{lemma:Cutoff}
    \begin{enumerate}
    \item The functor $\iota^\bR_{>0}$ is fully faithful.
        \item The image of the functor $\iota^\bR_{>0}$ is $\Im(\star_R\bK_{t\geq 0})$. Composing it with the inclusion equivalence $\Im(\star_R\bK_{t\geq 0})\hookrightarrow {}^\perp\Sh^\bR_{\tau\leq 0}(M\times \bR_t)$, it gives an inverse of the equivalence of Lemma~\ref{lem:quotientequiv}. 
        \item Composition of $\iota^\bR_{>0}\colon \mu(T^*M)\rightarrow\Sh^\bR(M\times \bR_t)$ and the quotient functor $q_\bR\colon \Sh^\bR(M\times \bR_t)\rightarrow \mu(T^*M)$ is the identity.
    \end{enumerate}
\end{lemma}
\begin{proof}
(1). The composition $\iota^\bR_{>0}\circ q_\bR|_{{}^\perp\Sh^\bR_{\tau\leq 0}(M\times \bR_t)}$ is $(-)\star_R \bK_{t\geq 0}$. This is an equivalence onto $\Im(\star_R\bK_{t\geq 0})$ by Lemma~\ref{lem:idem} and Lemma~\ref{lem:SOD}. Hence $\iota^\bR_{>0}$ is an equivalence. 

(2). Composing further with the inclusion equivalence $\Im(\star_R\bK_{t\geq 0})\hookrightarrow {}^\perp\Sh^\bR_{\tau\leq 0}(M\times \bR_t)$, the functor becomes $\star_R\bK_{t\geq 0}$. This is $\id$ by Lemma~\ref{lem:idem} and Lemma~\ref{lem:SOD}. Hence $\iota^\bR_{>0}$ is the inverse of $q_\bR|_{{}^\perp\Sh^\bR_{\tau\leq 0}(M\times \bR_t)}$. 

(3). By (2), we have
\begin{equation}
    q_\bR\circ \iota_{>0}\cong q_\bR|_{{}^\perp\Sh^\bR_{\tau\leq 0}(M\times \bR_t)}\circ \iota_{>0}\cong \id.
\end{equation}

\end{proof}

\subsection{Operations for positively microsupported equivariant sheaves}
We also would like to introduce the convolution product for positively microsupported sheaves.

We denote the composite functor
\begin{equation}
    q_\bR\circ (\iota^\bR_{>0}(-)\star_\bR\iota^\bR_{>0}(-)) \colon     (\mu(T^*M))^{\times 2} \rightarrow \mu(T^*M)
\end{equation}
by $\star^{>0}_\bR$. If the context is clear, we simply write it by $\star_\bR$.

We also denote the composition of functors
\begin{equation}
    q_\bR\circ (\iota^\bR_{>0}(-)\star_R\iota_{>0}(-))\colon      \mu(T^*M)\otimes \Sh(M\times \bR_t) \rightarrow \mu(T^*M).
\end{equation}
by $\star^{>0}_R$. If the context is clear, we simply write it by $\star_R$.

\subsection{Novikov ring action}
We consider $\bK_{t\geq 0}\in \Sh(M\times \bR_t)$. Then we obtain an object $1_{\mu M}:=q_\bR\circ \frakf^L(\bK_{t \geq 0})$.
\begin{lemma}\label{lem:Novcomputation}
    \begin{equation}
        H^0\End_{\mu(T^*M)}(1_{\mu M})\cong \Lambda_0.
    \end{equation}
\end{lemma}
\begin{proof}
Note that $\frakf^L(\bK_{t\geq 0})\in {}^\perp\Sh_{\tau\leq 0}^\bR(M\times \bR_t)$. We have
    \begin{equation}
        \begin{split}
            \Hom_{\mu(T^*M)}(1_{\mu M}, 1_{\mu M})&\cong \Hom_{{}^\perp\Sh_{\tau\leq 0}^\bR(M\times \bR_t)}(\frakf^L(\bK_{t\geq 0}), \frakf^L(\bK_{t\geq 0}))\\
            &\cong  \Hom_{\Sh^\bR(M\times \bR_t)}(\frakf^L(\bK_{t\geq 0}), \frakf^L(\bK_{t\geq 0}))\\
            &\cong \Hom_{\Sh(M\times \bR_t)}(\bK_{t\geq 0}, \frakf\frakf^L(\bK_{t\geq 0}))\\
            &\cong \Hom_{\Sh(M\times \bR_t)}(\bK_{t\geq 0}, \bigoplus_{c\in \bR}\bK_{t\geq c}).
        \end{split}
    \end{equation}

We first note that
\begin{equation}
    \Hom_{\Sh(M\times \bR_t)}(\bK_{t\geq 0}, \bK_{t\geq c})=\begin{cases}
    &\bK \text{ for $c\geq 0$,}\\
    &0 \text{ otherwise}.
\end{cases}
\end{equation}
Hence $\Hom_{\Sh(M\times \bR_t)}(\bK_{t\geq 0}, \bigoplus_{c\in \bR}\bK_{t\geq c})\cong \Hom_{\Sh(M\times \bR_t)}(\bK_{t\geq 0}, \bigoplus_{c\geq 0}\bK_{t\geq c})$. So our computation is reduced to compute the space of sheaf homomorphisms from $\bK_{t\geq 0}$ to $\bigoplus_{c\geq 0}\bK_{t\geq c}$, which is the same as the global section space of $\bigoplus_{c\geq 0}\bK_{t\geq c}$. 

For the purpose, we view $\bigoplus_{c\geq 0}\bK_{t\geq c}$ as the sheafification of the presheaf
\begin{equation}
    U\mapsto \bigoplus_{c\geq 0}(\bK_{t\geq c}(U)).
\end{equation}
The presheaf already satisfies the locality condition. Then any global section of the sheafification is given by an open covering $\bR_t=\bigcup_{i\in I}U_i$ with compatible sections $\{s_i\in\bigoplus_{c\in \bR}\bK_{t\geq c}(U_i)\}_{i\in I}$. If $|I|$ is finite, such a section arises from a section of the presheaf. So we consider the case when $|I|$ is not finite. Since $\bR$ is paracompact, we can assume $I= \bN$. Then, for each $i\in \bN$, the set
\begin{equation}
    \lc c\in \bR\relmid s_j^c\neq 0, j\leq i \rc
\end{equation}
is a finite set where $s_j^c$ is the component of $s_j$ in $\bK_{t\geq c}(U_j)$. It implies that there exists a section $s_i'\in \bigoplus_{c\geq 0}(\bK_{t\geq c}(\bigcup_{j\leq i}U_j))$ of the presheaf such that $s_i'|_{U_j}=s_j$ for any $j\leq i$. 

This observation implies that we have
\begin{equation}
    \Gamma(\bR_t, \bigoplus_{c\geq 0}\bK_{t\geq c})\cong \lim_{+\infty\leftarrow a} \bigoplus_{c\geq 0}(\bK_{t\geq c}((-a,a))).
\end{equation}
We also have an identification of vector spaces $\bigoplus_{c\geq 0}(\bK_{t\geq c}((-a,a)))\cong \bK[\bR_{\geq 0}]/T^a\bK[\bR_{\geq 0}]$. Hence we have $\Gamma(\bR_t, \bigoplus_{c\geq 0}\bK_{t\geq c})\cong \Lambda_0$. It is easy to check that this also gives a ring isomorphism. This completes the proof.
\end{proof}

We would like to construct the Novikov ring action on the homotopy category of $\mu(T^*M)$. 
\begin{lemma}\label{lem:id}
     The functor $(-)\star_\bR 1_{\mu M}$ is isomorphic to the identity on $\mu(T^*M)$.
\end{lemma}
\begin{proof}
It is enough to show that $(-)\star_\bR \frakf(\bK_{t\geq 0})$ is the identity on ${}^\perp\Sh_{\tau\leq 0}^\bR(M\times \bR_t)$. This is already proved in Lemma~\ref{lem:idem} and Lemma~\ref{lem:SOD}.
\end{proof}
By the functoriality of $\star_\bR$, Lemma~\ref{lem:Novcomputation}, and Lemma~\ref{lem:id}, we get a sequence of morphisms
\begin{equation}
    \Lambda_0\rightarrow \End_{\mu(T^*M)}(1_{\mu M})\rightarrow \End_{\mu(T^*M)}(\cE\star_\bR 1_{\mu M})\simeq \End_{\mu(T^*M)}(\cE)
\end{equation}
This gives a $\Lambda_0$-linear structure of the homotopy category of $\mu(T^*M)$. 

\begin{remark} Since $(-)\star 1_{\mu M}$ is isomorphic to the identity on $\mu(T^*M)$, the image of $(-)\star 1_{\mu M}$ is quasi-equivalent to $\mu(T^*M)$. The former dg-category is enriched over $\Lambda_0$ by the above observation, hence the latter category is also enriched over $\Lambda_0$ in a homotopical sense.
\end{remark}

\subsection{Non-conic microsupport}

Take an exact symplectic structure $\omega$ on $T^*M$ with its primitive $\lambda$. Then $(T^*M\times \bR_t, dt+\lambda)$ is a contact manifold.
We denote the standard symplectic structure of $T^*M$ by $\omega_{st}$ with its standard primitive $\lambda_{st}$. Then $\omega_{st}+d\tau\wedge dt$ is a symplectic structure of $T^*M\times T^*_{\tau>0}\bR_t$.
Suppose there exists an $\bR_{>0}$-action on $T^*M\times T^*_{>0}\bR_T$ which makes it a homogeneous symplectic manifold and its quotient is the contact manifold $(T^*M\times \bR_t, dt+\lambda)$. Then we have projections $T^*M\times T^*_{>0}\bR_t\rightarrow T^*M\times \bR_t\rightarrow T^*M$ where the second projection is the stupid projection. We denote the composition by $\rho$. By the uniqueness lemma of Viterbo~\cite{Vit}, such $\rho$ is unique for a given $\lambda$ if it exists.

The following is the fundamental example. 
\begin{example}\label{rho}
For the standard Liouville structure, $\rho_{st}:=\rho$ can be explicitly written as $\rho_{st}\colon T^*M\times T^*_{\tau>0}\bR_t\rightarrow T^*M$ by $(x, \xi, t, \tau)\mapsto (x, \xi/\tau)$ where $x\in M, \xi\in T^*_{x}M$.
\end{example}

In the following, we write $\rho_{st}$ by $\rho$ unless specified.
\begin{definition}
For an object $\cE\in \Sh^\bR_{\tau>0}(M\times \bR_t)$, we set
\begin{equation}
    \musupp(\cE):=\rho(\SS(\cE)\cap\{\tau>0\}).
\end{equation}
\end{definition}

\section{Sheaf quantization}
In this section, we discuss the microlocalization of sheaves in the equivariant context. The nonequivariant version is discussed in \cite{KS, Guillermou, JT}.

\subsection{Maslov covering and graded Lagrangian}
Let $V$ be a symplectic vector space. The fundamental group of the Lagrangian Grassmannian $LGr(V)$ is isomorphic to $\bZ$. We denote the universal covering by $\pi_V\colon \widetilde{LGr}(V)\rightarrow LGr(V)$, which is a $\bZ$-covering.

Let $X$ be a symplectic manifold. We denote the Lagrangian Grassmannian bundle by $LGr(X)$. 
\begin{definition}[{\cite[2b]{SeidelGraded}}]
    A Maslov covering $\widetilde{LGr}(X)$ is a fiber bundle over $X$ with a morphism $\widetilde{LGr}(X)\rightarrow {LGr}(X)$ such that it is the universal covering fiberwisely.
\end{definition}
Let $L$ be a Lagrangian submanifold of $X$. By taking its tangent fibers, we get a section of the projection $\lambda_L\colon LGr(X)|_L\rightarrow L$. We call the section the Lagrangian Gauss map.
\begin{definition}
A graded Lagrangian submanifold is a Lagrangian submanifold $L$ with a choice of a map $L\rightarrow \widetilde{LGr}(X)|_L$ which is a lift of the Lagrangian Gauss map.
\end{definition}

When $X$ is the cotangent bundle $T^*M$, there exists a canonical choice of a Maslov covering~\cite[2b]{SeidelGraded} as follows:
\begin{definition}
    A Lagrangian distribution over $X$ is a choice of subbundle of $TX$ such that it is a Lagrangian subspace in each fiber.
\end{definition}
Let $X$ be a symplectic manifold with a Lagrangian distribution. Then we get a section of $LGr(X)\rightarrow X$. Then we can construct the fiberwise universal covering of $LGr(X)$ by taking the image of $LGr(X)$ as fiberwise base point.

We consider the case when $X$ is a cotangent bundle $T^*M$. We denote the projection to the base $M$ by $\pi$. For any point $p\in T^*M$, the assignment $p\mapsto T_{f, p}T^*M:=T_{p}T^*_{\pi^{-1}\pi(p)}M\subset T_{p}T^*M$ forms a Lagrangian distribution. Then we can form a fiberwise universal covering as above. In the following, for cotangent bundles, we always use this Maslov covering.

\begin{example}[cotangent fiber]\label{ex:cotfibergrading}
    For a point $x\in M$, the cotangent fiber $L=T^*_xM$ is Lagrangian. The Gauss map $L\rightarrow LGr(T^*M)$ is in the Lagrangian distribution. Hence we can lift it to $L\rightarrow \widetilde{LGr}(T^*M)$ as a fiberwise trivial loop.
\end{example}
\begin{example}[Graph]\label{ex:graphgrading}
The zero section $T^*_MM$ is Lagrangian. We have a splitting $TT^*M|_{T^*_MM}=T_fT^*M\oplus TM$. Each factor is a Lagrangian subspace. Then the rotation action
\begin{equation}
    r_\theta:=\begin{pmatrix}
        \cos\theta &-\sin\theta\\
        \sin\theta &\cos\theta
    \end{pmatrix}\colon T_fT^*M\oplus TM\rightarrow T_fT^*M\oplus TM
\end{equation}
acts as a symplectic bundle isomorphism over $M$. In particular, $r_\theta(T_fT^*M)$ ($\theta\in [0, \pi/2]$) gives a fiberwise path in $LGr(T^*M)$ from the Lagrangian distribution to the Gauss map image of $T^*_MM$. Hence this gives a grading of $T^*_MM$. More generally, for a closed 1-form $\phi$, the graph $G(\phi)$ of $\phi$ is a Lagrangian submanifold of $T^*M$. For $c\in [0, 1]$, the graph of $c\cdot \phi$ gives a Lagrangian isotopy between $G(\phi)$ and $T^*_MM$. Hence the above grading of $T^*_MM$ induces a grading of $G(\phi)$.
\end{example}

\subsection{Relative Pin structure}
We also introduce the notion of relative Pin structure. Let $L$ be an $n$-dimensional manifold. Then the classification map of the tangent bundle is 
\begin{equation}
    L\rightarrow BO(n).
\end{equation}
The second Stiefel--Whitney class is given by the composition of the above morphism with the universal second SW class $BO(n)\rightarrow B^2\bZ/2$. We also have the following homotopy exact sequence:
\begin{equation}
    BPin^+(n)\rightarrow BO(n)\rightarrow B^2\bZ/2.
\end{equation}
Here $Pin^+(n)$ is a double cover of $O(n)$ with the center $\bZ/2\times\bZ/2$.

\begin{definition}
    A Pin structure on $L$ is a lift $L\rightarrow BPin^+(n)$ of the map $L\rightarrow BO(n)$.
\end{definition}
In other words, it is a null homotopy of the second SW map. From the above homotopy exact sequence, we need the vanishing of the second SW class to obtain a Pin structure.

We can consider a relative version. Fix $w\in H^2(L, \bZ/2)$. Then we can twist $L\rightarrow B^2\bZ/2$ by adding $-w$.
\begin{definition}
    A null-homotopy of the second SW map twisted by $w$ is called a relative Pin structure with the background class $w$.
\end{definition}

In the case of $X=T^*M$, we consider the background class $w_2(T^*M)|_L$.

\begin{example}[Graph]\label{ex:graphPin}
Suppose $X$ is the cotangent bundle $T^*M$. We consider the zero section $L=T^*_MM$. Then $w_2(X)|_L=w_2(L)$. Hence the twist gives the trivial map $L\rightarrow B^2\bZ/2$. Hence there exists a trivial Pin structure. Similarly, we can equip the graph of closed 1-form with a relative Pin structure.
\end{example}

\begin{example}[Cotangent fiber]\label{ex:cotfiberPin}
 We consider the cotangent fiber $L=T^*_xM$. Then $w_2(X)|_L=0=w_2(L)$. Hence the twist gives the trivial map. Hence there exists a trivial Pin structure. 
\end{example}

\subsection{Microsheaves along Lagrangian}
In this subsection, we recall the notion of microsheaves. The content in this subsection is essentially contained in \cite{Jin}.

For an open subset $U\subset T^*N$. we set
\begin{equation}
    \Sh(U; \bK_N):=\Sh(\bK_N)/\lc\cE\relmid \SS(\cE)\subset T^*N\bs U \rc.
\end{equation}
The assignment forms a prestack. The stackification is denoted by $\mu\Sh(-)$, called the Kashiwara--Schapira stack.

For a subset $A\subset U$, we can consider the subsheaf spanned by the objects supported in $A$. We denote it by $\mu\Sh_A$. This can be viewed as a sheaf supported on $A$. 

Let $L$ be a Lagrangian submanifold in the cotangent bundle $T^*M$. Restricting the Liouville form $\lambda$ to a contractible open subset $U$ of $L$, there exists a primitive of $\lambda|_U$ by the Poincare lemma. We fix such a primitive $f_U\colon U\rightarrow \bR$. Then we set
\begin{equation}
    L_{f_U}:=\lc (x, \xi, t, \tau)\in T^*M\times \{\tau>0\}\relmid (x, \xi/\tau)\in U, t=-f_U(x, \xi/\tau)\rc.
\end{equation}
Under the twisted projection $\rho:=\rho_{st}$, we have $\rho(L_{f_U})=U$. We can consider the sheaf $\mu\Sh_{L_{f_U}}$ on $L_{f_U}$. Since $\mu\Sh_{L_{f_U}}$ is conical, it descends to a sheaf on $U$. Also, the resulting sheaf on $U$ does not depend on the choice of $f$. Hence we denote the resulting sheaf by $\mu\Sh_{U}$.

Take an contractible open covering $\{U_i\}_{i\in I}$ of $L$. For each $U_i$, put $\mu\Sh_{U_i}$. Again, the sheaf is defined locally on $U_i$, they can be glued up on the intersections $U_i\cap U_j$. Hence we get a sheaf of categories $\mu\Sh_L$ over $L$. The sheaf is a locally constant sheaf of categories, hence classified by a map
\begin{equation}
    \mathrm{KS}\colon L\rightarrow B\Pic(\Mod(\bK)).
\end{equation}
where $B\Pic(\Mod(\bK))$ is the delooping of the Picard groupoid of $\Mod(\bK)$. Note that giving a simple global section of $\mu\Sh_L$ is equivalent to giving a null homotopy of $\mathrm{KS}$, and further equivalent to giving an equivalence $\mu\Sh_L\cong \mathrm{Loc}_L$ where $\mathrm{Loc}_L$ is the sheaf of the local systems over $L$.
\begin{definition}
    A $\bK$-brane structure of $L$ is a null homotopy of $\mathrm{KS}$.
\end{definition}

On the other hand, we have the Lagrangian Gauss map $L\rightarrow U/O$ to the stable Lagrangian Grassmannian and the delooped $J$-homomorphism $B(U/O)\rightarrow \Mod(\bS)$ where the RHS is the module category of the sphere spectrum.

\begin{theorem}
    If $\bK=\bZ$, the map $KS$ is given by $L\rightarrow B\bZ\times B^2\bZ/2\bZ$ where the first factor is the Maslov class and the second factor is the relative Stiefel--Whitney class.
\end{theorem}
\begin{proof}
    This is the result of \cite{guillermou2015quantization} for the case $\bK$ is a ring (see \cite{Jin} for more general coefficients). Although the statement in \cite{guillermou2015quantization} is only for exact case, since the statement is essentially local, it works for nonexact case.
\end{proof}

Since $\bZ$ is initial among the rings, we obtain the following.
\begin{corollary}
    For any commutative ring $\bK$, the grading and a relative Pin structure of $L$ gives a $\bK$-brane structure.
\end{corollary}

We define a full subcategory as
\begin{equation}
    \mu_L(T^*M):=\lc\cE\in \mu(T^*M)\relmid \musupp(\cE)\subset L \rc.
\end{equation}
In the following sections, we will construct the microlocalization functor
\begin{equation}
    \mu_L\colon \mu_L(T^*M)\rightarrow \mu\Sh_L(L).
\end{equation}

\subsection{Microlocalization}
We consider the Kashiwara--Schapira stack $\mu \mathrm{Sh}(-)$. For a Lagrangian submanifold $L$, we can consider the substack $\mu \mathrm{Sh}_{\rho^{-1}(L)}(-)$ consisting of objects supported on $\rho^{-1}(L)$. 

Take an open contractible covering $\{U_i\}$ of $L$ i.e., $U_i$ is open in $T^*M$ and $U_i\cap L$ is contractible and any intersections of $U_i$ also satisfies the same assumption.
For each open subset $U_i$, we have a presentation $U_i\cap \rho^{-1}(L)=\bigcup_{c}L_{f+c}\cap U_i$ where $f$ is some primitive of $\lambda|_L$.
\begin{lemma}
  For an object $\cE\in \mu\mathrm{Sh}_{\rho^{-1}(L)}(U)$, the restriction $\cE|_{L_{f+c}}$ is in $ \mu\mathrm{Sh}_{L_{f}}|_{L_f}(U\cap L_f)$. Here $U$ is some intersection of $U_i$'s.
\end{lemma}
\begin{proof}
    Then there exists a local symplectomorphism $\phi$ such that $L$ is mapped to the zero section. By taking the contactization, we get a local contactomorphism (or equivalently, a local homogeneous symplectomophism) $\widetilde{\phi}$ mapping $L_{f+c}$ to $\bR^n_x\times \bR_{\tau>0}\subset T^*\bR^n_x\times\{c\}\times T^*\bR_t$. By the quantized contact transformation~\cite{KS}, this gives an equivalence $\mu\Sh_{\rho^{-1}(L)}(U)\cong \mu\Sh_{\widetilde{\phi}(\rho^{-1}(L))}(\widetilde{\phi}(U))$. Then the latter category is given by the sheafification of quotient categories of
    \begin{equation}
        \lc \cE\in \Sh_{\tau>0}(M\times \bR_t)\relmid \SS(\cE)\subset \bR^n_x\times \bR_t\times \bR_{\tau>0}\rc.
    \end{equation}
    Since any object of this category is represented by an object of the form $p_M^{-1}\cE_0$ where $p_M\colon \bR^n_x\times \bR_t\rightarrow \bR_t$, the microstalks along $\bR^n_x\times \{t=0\}\times \bR_{\tau>0}$ is constant. This completes the proof.
\end{proof}

In the next section, we will prove this lemma by using quantized contact transformation.

Let $C(U)$ be the Cech covering associated to $\{U_i\}$.
We have a sequence of functors
\begin{equation}
    \mu_L(T^*M)\xrightarrow{\frakf} \Sh_{\rho^{-1}(L)}(M\times \bR_t)\rightarrow \prod_{U\in C(\cU)}\mu\mathrm{Sh}_{\rho^{-1}(L)}(U)\rightarrow \prod_{U\in C(\cU)}\mu\mathrm{Sh}_{L_{f}}|_{L_f}(U\cap L_f)\rightarrow \mu\mathrm{Sh}(L)
\end{equation}
where the leftmost morphism is the colimit with respect to the Cech covering, and $\Sh_{\rho^{-1}(L)}(M\times \bR_t)$ is the full subcategory of $\Sh(M\times \bR)$ consisting of the objects whose microsupports are in the zero section and $\rho^{-1}(L)$. The composition is our desired microlocalization functor, and will be denoted by $\mu_L$.

\subsection{Brane structure and microlocalization}
Now we can define the notion of Lagrangian brane.
\begin{definition}
    A Lagrangian brane $(L, \alpha, b, \cL)$ is a tuple of the following data:
    \begin{enumerate}
        \item a graded Lagrangian submanoifold $L$ with $\alpha\colon L\rightarrow \widetilde{LGr}(X)|_L$
        \item a relative Pin structure $b$ of $L$
        \item a derived local system $\cL$ over $L$.
    \end{enumerate}
    When $\cL$ is the rank 1 constant local system, we simply say that $(L,\alpha, b)$ is a Lagrangian brane.
\end{definition}
\begin{remark}
    There are further generalizations of the notion of Lagrangian branes. See e.g. \cite{JT}.
\end{remark}

Given a Lagrangian brane $(L, \alpha, b)$, we have an equivalence $\mu\Sh_L(L)\cong \mathrm{Loc}(L)$. Composing it with $\mu_L$, we obtain 
\begin{equation}
    \mu_{(L, \alpha,b)}\colon \mu_L(T^*M)\rightarrow \mathrm{Loc}(L).
\end{equation}

\begin{definition}
    Let $(L, \alpha, b, \cL)$ be a Lagrangian brane. A sheaf quantization of $(L, \alpha, b, \cL)$ is an object of $\mu_L(T^*M)$ such that $\mu_{(L, \alpha, b)}(\cE)\cong \cL$.
\end{definition}
More simply,
\begin{definition}
\begin{enumerate}
    \item Let $L$ be a Lagrangian submanifold. A sheaf quantization of $L$ is an object of $\mu_L(T^*M)$. 
    \item Let $L$ be a Lagrangian submanifold. A pure sheaf quantization of $L$ is an object of $\mu_L(T^*M)$ such that $\mu_{(L, \alpha, b)}(\cE)$ is concentrated in a single degree for some brane structure $(L, \alpha, b)$.
    \item Let $L$ be a Lagrangian submanifold. A simple sheaf quantization of $L$ is a pure sheaf quantization whose microlocalization is rank 1.
\end{enumerate}
\end{definition}

We also sometimes use the following terminology.
\begin{definition}
    Let $(L, \alpha, b)$ be a Lagrangian brane. The microstalk of an object $\cE$ of $\mu_L(T^*M)$ at $(x, \xi)\in L$ is the stalk of $\mu_{(L,\alpha, b)}(\cE)$ at $(x, \xi)$. Note that the definition does not depend on $b$. Moreover if $L$ is connected, the isomorphism type of microstalk does not depend on $(x, \xi)\in L$.
\end{definition}

\begin{remark}[Uniqueness]
In the nonexact setting, the present definition of brane structure does not characterize sheaf quantization. To rigidify it, we have to include bounding cochains. We will treat it in another publication.
\end{remark}

\begin{remark}[Sheaf quantization in $\Sh(M)$ and $\Sh_{\tau>0}(M\times\bR_t)$]

(Version 0). For a conic Lagrangian submanifold $L$ in $T^*M$, a sheaf quantization $\cE$ of $L$ is a constructible sheaf with $\SS(\cE)=L$. 

(Version 1). If $L$ is an exact Lagrangian submanifold, we can take a primitive of $\lambda_{st}|_L$ globally on $M$. Then we can lift $L$ to $L_f\subset T^*M\times T^*_{>0}\bR_t$. An object $\cE\in \Sh_{\tau>0}(M\times \bR_t)$ on $M\times \bR_t$ is said to be a sheaf quantization of $L$ if $\SS(\cE)\bs T^*_{M\times \bR_t}M\times \bR_t= L_f$ and has finite-dimensional pure microstalks. Under certain assumptions on $L$, Guillermou and Jin--Treumann constructed such sheaves~\cite{Guillermou, JT}. This notion is a generalization of Version 0 in the following sense: a sheaf quantization $\cE$ of a conic Lagrangian $L$ gives a sheaf on $M\times \bR_t$ by $\cE\boxtimes \bK_{[0,\infty)}$, which turns out to be a sheaf quantization (Version 1) of $L$.

(Version 2). Our concept of sheaf quantization generalizes both. For a sheaf quantization (Version 1) $\cE\in \Sh_{\tau>0}(M\times \bR_t)$ of an exact Lagrangian $L$, consider $\bigoplus_{c\in \bR}T_c\cE$ with the obvious $\bR_t$-equivariant structure. Then this is a sheaf quantization of $L$ in our sense.
\end{remark}

\begin{example}[Graph]\label{ex:graphmicrolocalization}
    Let $f$ be a smooth function $M\rightarrow \bR$. We will denote the graph of the differential $df$ by $L$. It is well-known that the microsupport of the sheaf $\bK_{t\geq -f(x)}$ can be computed as
    \begin{equation}
        \SS(\bK_{t\geq -f(x)})\cap \{\tau>0\}=L_f.
    \end{equation}
    The direct sum $\cE_{f}:=\bigoplus_{c\in \bR}\bK_{t\geq -f(x)+c}=\frakf^L(\bK_{t\geq -f(x)})$ has $\musupp(\cE_f)=L$. As we have seen in Example~\ref{ex:graphgrading} and Example~\ref{ex:graphPin}, $L$ is canonically equipped with a brane structure. 

    We would like to compute the microlocalization. The projection $\pi\colon L_f\rightarrow M$ is a diffeomorphism. By the Weinstein theorem, we have a symplectomorphism between a neighborhood of $L_f$ and a neighborhood of $M$ induced by the projection. Quantizing this symplectomorphism, $\cE_f$ is mapped to $\bigoplus_{c\in \bR}\bK_{t\geq c}$.

    By using the standard brane structure of the zero section, the local section of the microlocalization over $U\subset M$ is given by
    \begin{equation}
        \Gamma_{\lc t\geq 0\rc}(U\times \{t=0\}, \bigoplus_{c\in \bR}\bK_{t\geq c})\cong \lim_{\substack{\longrightarrow\\ c\rightarrow +0}}\Hom_{\Sh(U\times \bR)}(\bK_{c>t\geq 0}, \bigoplus_{c\in \bR}\bK_{t\geq c}).
    \end{equation}
    By the computation of the proof of Lemma~\ref{lem:Novcomputation}, the above colimit in $H^0$ is given by 
    \begin{equation}
        \lim_{\substack{\longrightarrow\\ c\rightarrow +0}}H^0(\Hom_{\Sh(U\times \bR)}(\bK_{c>t\geq 0}, \bigoplus_{c\in \bR}\bK_{t\geq c}))\cong\lim_{\substack{\longrightarrow\\ c\rightarrow +0}}\Lambda_0/T^c\Lambda_0\cong \bK.
    \end{equation}
    Hence $H^0$ of the microlocalization is $\bK_L$. 
\end{example}

\section{Meromorphic flat $\hbar$-connection}
\subsection{Meromorphic flat $\hbar$-connection}
In this section, we would like to set up our general setting for differential equations. 

Let $C$ be a Riemann surface. Let $\cO_C$ be the structure sheaf. We consider the subring $\cD^\hbar_C$ of the sheaf of $\bC[\hbar]$-linear endomorphisms of $\cO_C[\hbar]:=\cO_C\otimes_\bC\bC[\hbar]$ generated by $\hbar\partial_z$ and $\cO_C$. We would like to consider the modules over $\cD^\hbar_C$.

Our fundamental examples of such modules are given by meromorphic flat $\hbar$-connections:
\begin{definition}
A meromorphic flat $\hbar$-connection is given by the following data
\begin{enumerate}
    \item A meromorphic bundle $\cE$ with poles in $M\subset C$ where $M$ is a finite subset of $C$. In other words, a locally free $\cO_C(*M)$-module where $\cO_C(*M)$ is the sheaf of meromorphic functions with possible poles in $M$.
    \item A $\bC[\hbar]$-linear morphism $\nabla\colon \cE[\hbar]\rightarrow \cE[\hbar]\otimes_{\cO_C} K_C$ where $\cE[\hbar]:=\cE\otimes_\bC\bC[\hbar]$ and $K_C$ is the canonical bundle which satisfies the following: For any $f\in \cO_C$ and $s\in \cE$, we have $\nabla(fs)=s\otimes \hbar df+ f\nabla s$. 
\end{enumerate}
\end{definition}
A meromorphic flat $\hbar$-connection $(\cE, \nabla)$ gives a $\cD^\hbar$-module as follows: The underlying $\cO_C[\hbar]$-module is $\cE[\hbar]$ and the action of $\cD^\hbar$ is defined as follows: For a tangent vector $v\in \cD$, we have $\hbar v\in \cD^\hbar$. For $s\in \cE$, we set
\begin{equation}
    (\hbar v)\cdot s:=\nabla s(v).
\end{equation}

For $\hbarr\in \bC^\times$, we define a meromorphic connection $\cM^\hbarr:=(\cE,\nabla^\hbarr)$ by 
\begin{equation}
    \nabla^\hbarr:=\frac{1}{\hbarr}\cdot \nabla\otimes_{\bC[\hbar]}\bC
\end{equation}
where the morphism $\bC[\hbar]\rightarrow\bC$ of the tensor is given by the evaluation at $\hbarr$.

\begin{remark}
The notion of $\hbar$-connection is almost the same as Deligne's $\lambda$-connections. However, we would like to avoid the terminology in this paper, because of the existence of a non-straight-forward relations between these when one starts from a point in Hitchin's base~\cite{GaiottoTBA, Mulaseetal}.
\end{remark}

\subsection{Spectral curve}
It is well-known that $\cD^\hbar_C$ is a deformation quantization of $T^*C$. Namely, there exists a canonical isomorphism $\cD_C^\hbar/\hbar\cD_C^\hbar\cong \cO_{T^*C}$. Let $\cM$ be a $\cD^\hbar_C$-module. Then $\cM/\hbar\cM$ is a module over $\cO_{T^*C}$. We set
\begin{equation}
    \SS(\cM):=\supp(\cM/\hbar\cM)\subset T^*C.
\end{equation}
We sometimes call it the {\em spectral curve} of $\cM$. The spectral curve is holomorphic coisotropic with respect to the standard symplectic structure of $T^*C$. In the rest of this paper, we will only consider $\hbar$-connections whose spectral curve is holomorphic Lagrangian.

\section{Planck constant in complex and real symplectic geometry}
Microlocal sheaf theory is related to real symplectic geometry. On the other hand, the concept of $\hbar$-connection is related to deformation quantizations of complex symplectic manifolds. In this section, we clarify the relationship.

\subsection{Complex symplectic geometry and deformation quantization}
Let $X$ be a complex manifold. Then the cotangent bundle $T^*X$ is a complex exact symplectic manifold. The canonical symplectic form $\omega$ can be locally written as $\sum d\zeta_i\wedge dz_i$. Here $z_i$ is a local coordinate of $X$ and $\zeta_i$ is the corresponding cotangent coordinate. The form $\omega$ is exact and has a canonical primitive $\lambda$, locally written as $\sum \zeta_i dz_i$.

We will consider the canonical quantization of $T^*X$ from local pieces. The canonical deformation quantization of $T^*X$ is given by the ring $\cD^\hbar:=\bC[\hbar][z_i, \hbar \partial_{z_i}]$. In the classical limit, $\hbar\partial_z$ corresponds to $\zeta$. 

This gives a hint to make a Lagrangian into a conic Lagrangian: The original coordinate $\zeta$ splits into the product $\hbar\partial_z$ in the setting of quantization. Hence we will reconsider $\hbar$ and $\partial_z$ as independent variables. Then a Lagrangian submanifold $L\subset T^*M$ lifts to 
\begin{equation}
    \widetilde{L}:=\lc (z, p,\hbar)\in T^*X\times \bC^\times_\hbar \relmid (z, \hbar p)\in L\rc.
\end{equation}
For the later use, it is convenient to work with $\eta:=\hbar^{-1}$. To get a Lagrangian submanifold, we moreover consider the Fourier dual coordinate $y$ of $\eta$. We set
\begin{equation}
    {L}^\bC_f:=\lc (z, p, \hbar, y)\in T^*X\times \bC^\times_\hbar\times \bC_y \relmid (z, \hbar p)\in L, y= -f(z, \hbar p)\rc
\end{equation}
where $f$ is a primitive of $L$ i.e., $df=\lambda|_L$.
We consider $\bC_y\times \bC^\times_\hbar\cong T^*_{\eta\neq 0}\bC_y$. Then ${L}^\bC_f$ is a Lagrangian manifold in $T^*X\times T^*_{\eta\neq 0}\bC_y$ with respect to the standard symplectic structure and is conic i.e., invariant under the scaling action of $\bC^\times_\hbar$ on the fibers of $T^*X\times T^*_{\eta\neq 0}\bC_y$. 

There exists a projection
\begin{equation}
    \rho^\bC\colon T^*X\times T^*_{\eta\neq 0}\bC_y\rightarrow T^*X; (z, \zeta, y, \eta)\mapsto (z, \zeta/\eta)=(z, \hbar\zeta).
\end{equation}
Then the image of $L^\bC_f$ is precisely $L$. Conversely, $L^\bC_f$ is a Lagrangian leaf of $(\rho^\bC)^{-1}(L)$. This procedure is a classical trick in the theory of deformation quantization (for example, \cite{PS}).

\subsection{From complex symplectic geometry to a family of real symplectic geometries}

Considering $X$ as a real manifold and denote it by $X_\bR$. We have $T^*X_\bR$ the real cotangent bundle equipped with the standard real symplectic structure $\omega_{st}$. 

In a holomorphic local coordinate $\{z_i\}$ of $X$, we have the associated real coordinate $(x_i, y_i)$ with $z_i=x_i+\sqrt{-1}y_i$. We can further identify the real and complex cotangent bundles via
\begin{equation}
dx_i\mapsto dz_i, dy_i\mapsto -\sqrt{-1}dz_i.
\end{equation}
Hence in the cotangent coordinate, we have $(\xi_i, \eta_i)\mapsto \xi_i-\sqrt{-1}\eta_i=\zeta_i$. This implies that $\Re(\omega)=\omega_{st}$. In the following, we will not distinguish $T^*X$ with $T^*X_\bR$.

Hence a holomorphic Lagrangian submanifold with respect to $\omega$ gives a real Lagrangian submanifold with respect to $\omega_{st}$.

Now we are going to vary the holomorphic symplectic form. For $\hbarr\in\bC^\times_\hbar$, we have another holomorphic symplectic form $\omega/\hbarr$. This induces a new real symplectic form $\Re(\omega/\hbarr)$ on $T^*X$. A holomorphic Lagrangian submanifold with respect to $\omega$ is also a holomorphic Lagrangian with respect to $\omega/\hbarr$. Hence it is also a real Lagrangian submanifold with respect to $\Re(\omega/\hbarr)$.

If $L$ is a holomorphic exact Lagrangian with respect to $\lambda$, then $L$ is also a real exact Lagrangian with respect to $\Re(\lambda/\hbarr)$. In this case, a real primitive function $f_\hbarr\colon L\rightarrow \bR$ of $\Re(\lambda/\hbarr)|_L$ can be also considered as a real primitive function of $L/\hbarr$ with respect to $\Re(\lambda)=\lambda_{st}$.

\begin{remark}
Since microlocal sheaf theory knows only about the standard symplectic form, to speak about $\Re(\omega/\hbarr)$, we will consider the following identification.
\begin{equation}
    (T^*X, \Re(\omega/\hbarr))\cong (T^*X, \Re(\omega/\hbarr))\xrightarrow{\cdot (1/\hbarr)} (T^*X, \Re\omega)\cong (T^*X, \omega_{st})
\end{equation}
Hence a sheaf quantization of a real Lagrangian submanifold $L$ on the leftmost side is interpreted as a sheaf quantization of $L/\hbarr$ in the rightmost side.
\end{remark}

For the Liouville form $\Re(\lambda/\hbarr)$, the associated $\rho$ can be written as follows:
\begin{equation}
\rho_\hbarr\colon T^*X\times T^*_{\tau>0}\bR_t\rightarrow  T^*X; (z, \zeta, t, \tau)\mapsto (z, \hbarr\zeta/\tau).
\end{equation}
If the context is clear, we will omit the subscript $\hbarr$.

We can compare this picture with the story of deformation quantization. We will consider $\bR_{\tau>0}$ is a ray in $\bC^\times_\eta$ defined by $\hbarr^{-1}\bR_{\tau>0}$. Then we set the following morphism
\begin{equation}
\begin{split}
    \iota\colon T^*X\times T^*_{\tau>0}\bR_t\rightarrow T^*X\times T^*_{\eta \neq 0}\bC_y; (z, \zeta, t, \tau)\mapsto (z, \zeta, \hbarr t, \hbarr^{-1}\tau).
\end{split}
\end{equation}
Then $\rho^\bC\circ \iota=\rho_\hbarr$. This means $\rho_\hbarr$ precisely corresponds to a rotation of $\rho_{st}$ in Example~\ref{rho} by $\hbarr$.

\subsection{Sheaf quantizations in complex cotangent bundles}
Now we redefine sheaf quantization at $\hbarr\in \bC^{\times}_{\hbar}$ as follows.
\begin{definition}
Let $L$ be a holomorphic Lagrangian submanifold of $T^*X$. An object $\cE$ of $\Sh^\bR_{\tau>0}(X\times \bR_t)$ is a {\em sheaf quantization} of $L$ at $\hbarr$ if it is a sheaf quantization of $L/\hbarr$.
\end{definition}
If one uses $\rho_\hbarr$ in the definition of of $\musupp$, we obtain a variant of $\musupp$. We denote it by $\musupp_\hbarr$. Then, for a sheaf quantization $\cE$ of $L$ at $\hbarr$, we have $\musupp_\hbarr(\cE)=L$. We denote the full subcategory of $\Sh^\bR_{\tau>0}(X\times \bR_t)$ spanned by the objects with $\musupp_\hbarr(\cE)\subset L$ by $\mu_{L, \hbarr}(T^*M)$. Of course, $\mu_{L, \hbarr}(T^*M)=\mu_{L/\hbarr}(T^*M)$

Fix a path from $1$ to $\hbarr$ in $\bC^\times$. It gives a Lagrangian isotopy from $L$ to $L/\hbarr$. Hence, if $(L, \alpha, b, \cL)$ is a Lagrangian brane, $L/\hbarr$ is also canonically equipped with a Lagrangian brane structure $(L, \alpha_\hbarr, b_\hbarr, \cL_\hbarr)$. Hence we have the microlocalization functor $\mu_{(L, \alpha_\hbarr, b_\hbarr, \cL_\hbarr)}\colon \mu_{L, \hbarr}(T^*M)\rightarrow \mathrm{Loc}(L/\hbarr)$. On the other hand, we have $\mathrm{Loc}(L)\cong \mathrm{Loc}(L/\hbarr)$ by the above isotopy. By the composition, we get a functor $\mu_{(L, \alpha, b, \cL),\hbarr}\colon \mu_{L, \hbarr}(T^*M)\rightarrow \mathrm{Loc}(L)$.

\section{Sheaf quantizations associated to meromorphic flat $\hbar$-connections of rank 0 and 1}
In this section, as a warm-up, we would like to construct sheaf quantizations associated to rank 0 and rank 1 connections. We do not need exact WKB analysis here. The material of this section can be applied to any dimensions after minor modifications. In this section, $\bK=\bC$.

\subsection{What will we do?}
For a given $\cD^\hbar_C$-module $\cM$ and $\hbarr\in \bC^\times_\hbar$, what we will do here is to construct a sheaf quantization $S^\hbarr_\cM$ such that
\begin{enumerate}
    \item $\musupp(S^\hbarr_\cM)=\SS(\cM)$, and
    \item $S^\hbarr_\cM$ is constructed out of the  solutions of $\cM$.
\end{enumerate}
The second requirement is not mathematically well-defined, but one can see the meaning from the construction below. The sheaf quantization we will construct below is ``correct", since it is partially recovered from $\hbar$-Riemann--Hilbert correspondence~\cite{hRH}.

\subsection{Rank 0}
For a point $z\in C$, the cotangent fiber $T^*_zC$ is a Lagrangian submanifold. We would like to consider deformation quantization and sheaf quantization of $T^*_zC$. Note that $T^*_zC$ is canonically equipped with a grading $\alpha_{triv}$ (Example~\ref{ex:cotfibergrading}) and a relative Pin structure $b_{triv}$ (Example~\ref{ex:cotfiberPin}). 

Hence, a finite set of points $\{z_i\}_{i\in I}\subset C$, the union $L:=\bigcup_{i\in I}T^*_{z_i}C$ is canonically equipped with a grading $\alpha_{triv}$ and a relative Pin structure $b_{triv}$.  With a constant rank 1 local system $\bK_L$, we consider $L$ as a Lagrangian brane. We then have the associated microlocalization functor $\mu_{(L, \alpha_{triv}, b_{triv}, \bK_L)}\colon \mu_{L}(T^*C)\rightarrow \mathrm{Loc}(L)$. Note that $L=L/\hbarr$ in this case.

A connection of rank 0 means a $\cD^\hbar_C$-module supported on a 0-dimensional subvariety. Hence it is a direct sum of the form
\begin{equation}
    \bigoplus_{i\in I}\cD^\hbar_C/(z-z_i)^{n_i}\cD^\hbar_C\cong \bigoplus_{i\in I}(\cD^\hbar_C/(z-z_i)\cD^\hbar_C)^{\oplus n_i}.
\end{equation}
Here $\{z_i\}_{i\in I}$ is a subset of $C$.

The ``solution" of the equation $\cD^\hbar_C/(z-z_0)\cD^\hbar_C$ should be the delta function. Then it is canonical to associate a sheaf
\begin{equation}
    \bigoplus_{c\in \bR} \bK_{z_0}\boxtimes \bK_{[c, \infty)}\in \Sh^\bR_{\tau>0}(C\times \bR_t)
\end{equation}
with a canonical equivariant structure, to $\cD_C^\hbar/(z-z_0)\cD_C^\hbar$. Here $\bK_{z_0}$ is the skyscraper sheaf on $z_0$.
It is easy to see that $\musupp_\hbarr(\bigoplus_{c\in \bR} \bK_{z_0}\boxtimes \bK_{[c, \infty)})=T^*_{z_0}C$ which coincides with $\SS(\cD_C^\hbar/(z-z_0)\cD_C^\hbar)=T^*_{z_0}C$.

Consequently, we get the bijection of the form
\begin{equation}
    \bigoplus_{i\in I}\cD^\hbar_C/(z-z_i)^{n_i}\cD^\hbar_C \mapsto  \bigoplus_{i\in I}\bigoplus_{c\in \bR}\lb \bK_{z_i}\boxtimes \bK_{[c, \infty)}\in \Sh^\bR_{\tau>0}(C\times \bR_t)\rb^{\oplus n_i}.
\end{equation}

We can summarize the consequence as follows:
\begin{proposition}
Fix $\hbarr\in \bC_\hbar$. For any meromorphic $\hbar$-connection $\cM$ of rank $0$, there exists a sheaf quantization $\cS^\hbarr_\cM$ such that 
\begin{enumerate}
    \item $\musupp_\hbarr(\cS^\hbarr_\cM)=\SS(\cM)$, and  
    \item the microlocalization $\mu_{(L, \alpha_{triv}, b_{triv}, \bK_L)}(\cS^\hbarr_\cM)$ at $z$ is pure and its rank is the same as the rank of $\cM$ at $z$ for any $z\in C$.
\end{enumerate}
\end{proposition}
\begin{proof}
    We set 
    \begin{equation}
        \cS^\hbarr_\cM:=\bigoplus_{i\in I}\bigoplus_{c\in \bR}\lb \bK_{z_i}\boxtimes \bK_{[c, \infty)}\in \Sh^\bR_{\tau>0}(C\times \bR_t)\rb^{\oplus n_i}.
    \end{equation}
We have already seen that this satisfies the first condition. It is also easy to see the second condition.
\end{proof}

\subsection{Rank 1}
Let $\cM$ be a rank 1 meromorphic flat $\hbar$-connection with poles in $M\subset C$.
On a sufficiently small open subset $U$ in $C$, we can write the equation of flat section of $\cM$ as 
\begin{equation}
    (\hbar\partial-Q(z,\hbar))\psi=0
\end{equation}
where $Q(z, \hbar)=\sum_{i=0}^kQ_i(z)\hbar^i$ for some finite $k$ and each $Q_i$ is a meromorphic function. In $T^*U,$ the spectral curve $\SS(\cM)$ is defined by $\xi=Q_0$. In the following, we set $L:=\SS(\cM)\cap T^*(C\bs M)$.

The meromorphic function $Q_0$ is globally a meromorphic 1-form. Hence the spectral curve 
 $L$ is canonically equipped with a grading (Example~\ref{ex:graphgrading}) and a relative Pin structure (Example~\ref{ex:graphPin}).

For $\hbarr\in \bC^\times$, by taking a path from 1 to $\hbarr$, we obtain a Lagrangian isotopy from $L$ to $L/\hbarr$. Hence we obtain a grading $\alpha_h$ and a relative Pin structure $b_h$ on $L/\hbarr$. With a rank 1 constant local system $\bK_{L/\hbarr}$, we have the associated microlocalization functor $\mu_{(L, \alpha_h, b_h, \bK_L),\hbarr}\colon \mu_{L,\hbarr}(T^*C)\rightarrow \mathrm{Loc}(L)$.

Take an open covering $\{U_i\}$ of $C\bs M$ such that each intersection is contractible. Fix $\hbarr\in \bC^\times$ and a base point $z_i$ in each $U_i$. We can solve the equation explicitly:
\begin{equation}
    \psi(z,\hbarr)=\exp\lb \int_{z_i}^z\frac{Q(z,\hbarr)}{\hbarr}dz\rb
\end{equation}
On each $U_i\times \bR_t$, we consider the equivariant sheaf $\bigoplus_{c\in \bR}T_c\bK_{t\geq -\Re\int_{z_i}^z \frac{Q_0}{\shbarr}}$ where the equivariant structure is the obvious one.

We have the following:
\begin{lemma}\label{lem:microsuppestimate1}
    $\musupp_\hbarr (\bigoplus_{c\in \bR}T_c\bK_{t\geq -\Re\int_{z_i}^z \frac{Q_0}{\shbarr}})=L\cap T^*U_i$. Moreover, the microlocalization along $L\cap T^*U_i$ is the constant sheaf.
\end{lemma}
\begin{proof}
    By Example~\ref{ex:graphmicrolocalization}, we have $\musupp_\hbarr (\bigoplus_{c\in \bR}T_c\bK_{t\geq -\Re\int_{z_i}^z \frac{Q_0}{\shbarr}})=(L\cap T^*U_i)/\hbarr$ and the microlocalization is the constant sheaf. Multiplying $\hbarr$, we obtain the desired result. 
\end{proof}

On each overlap $U_i\cap U_j$, we consider the isomorphism between $\bigoplus_{c\in \bR}T_c\bK_{t\geq -\Re\int_{z_i}^z \frac{Q_0}{\shbarr}}$ and $\bigoplus_{c\in \bR}T_c\bK_{t\geq -\Re\int_{z_j}^z \frac{Q_0}{\shbarr}}$ induced by
\begin{equation}\label{eq:gluing}
    \exp\lb\int_{z_i}^{z_j}\frac{Q(z,\hbarr)}{\hbarr}dz\rb\cdot \id\colon \bK_{t\geq -\Re\int_{z_i}^z \frac{Q_0}{\shbarr}}\rightarrow \bK_{t\geq -\Re\int_{z_i}^z \frac{Q_0}{\shbarr}}=T_{-\Re\int_{z_i}^{z_j} \frac{Q_0}{\shbarr}}\bK_{t\geq -\Re\int_{z_j}^z \frac{Q_0}{\shbarr}}.
\end{equation}
By Lemma~\ref{lem:microsuppestimate1}, these gluing-up isomorphisms give a sheaf quantization of $L$ at $\hbarr$. 

If one chooses another base point $z_i'$ for $U_i$, one can identify the corresponding sheaves $\bigoplus_{c\in \bR}T_c\bK_{t\geq -\Re\int_{z_i}^z \frac{Q_0}{\shbarr}}$ and $\bigoplus_{c\in \bR}T_c\bK_{t\geq -\Re\int_{z_i'}^z \frac{Q_0}{\shbarr}}$ in the same way as in (\ref{eq:gluing}) replacing $z_j$ with $z_i'$. In this sense, the constructions here does not depend on the choices of base points.

On the other hand, the localization of $\cM^\hbarr$ on $C\bs M$ defines a local system 
 $\mathrm{Sol}(\cM^\hbar)$ on $C\bs M$. Under the projection, we have an equivalence $\pi_*\colon \mathrm{Loc}(L)\rightarrow \mathrm{Loc}(C\bs M)$.
\begin{proposition}\label{thm:rank1}
Given a rank 1 meromorphic $\hbar$-connection $\cM$ as above, there exists a sheaf quantization $S_\cM^\hbarr$ of $\SS(\cM)\cap T^*(C\bs M)$ at $\hbarr$ such that $\pi_*\mu_{(L, \alpha_h, b_h, \bK_L),\hbarr}(S_\cM^\hbarr)\cong \mathrm{Sol}(\cM^\hbarr)$.
\end{proposition}
\begin{proof}
The statement about $\musupp$ immediately follows from Example~\ref{ex:graphmicrolocalization}. Locally in $T^*U_i$, the microlocalization is given by $\bK_{L\cap T^*U_i}$ from Example~\ref{ex:graphmicrolocalization}. Each overlapping $L\cap T^*(U_i\cap U_j)$, we have two different identifications of the microlocalization with $\bK_{L\cap T^*(U_i\cap U_j)}$. They are related by the gluing morphism: the multiplication by $\exp\lb\int_{z_i}^{z_j}\frac{Q(z,\hbarr)}{\hbarr}dz\rb$. Hence we get the desired local system.
\end{proof}

\section{Schr\"odinger equations}
We would like to recall some basic facts of exact WKB analysis. We refer to \cite{KawaiTakei} and \cite{IwakiNakanishi} for more accounts. In this section, we concentrate on the case arising from Schr\"odinger equations.

\begin{remark}
Although it is probably not difficult to extend the result to more general rank 2 connections, some necessary results in exact WKB analysis is not available in the literature. Therefore we present our result in the simplest setup, which is enough to present the essence of the construction. The author's most general treatment in the previous version contains some gaps. Another approach can be found in \cite{hRH}.
\end{remark}
\subsection{Schr\"odinger equation and connection}
We consider a second order differential operator $\cQ$ on $C$ with a parameter $\hbar$ locally written as
\begin{equation}\label{globalschrodinger}
    \cQ=(\hbar\partial)^2-Q(z, \hbar).
\end{equation}
We call this operator {\em Schr\"odinger operator}. Here $Q(z,\hbar)=\sum_{i=0}^kQ_i(z)\hbar^i$ for some $k$ and each $Q_i$ is a meromorphic function. We fix a divisor $M$ such that each $Q_i$ is smooth outside $M$. Then $Q_0$ defines a global meromorphic quadratic differential on $C$ i.e., a global section of $K^{\otimes 2}(*M)$ for a divisor $M$.

To lift a Schr\"odinger operator to a flat $\hbar$-connection, fix a square root $\sqrt{K}$ of $K$, which is well-known to be equivalent to choose a spin structure. On $C\bs M$, we can define a connection by
\begin{equation}
    \cM:=\lb \sqrt{K}\oplus \sqrt{K}^{-1},  \hbar\partial+\begin{pmatrix}
    0& Q\\
    1 & 0
    \end{pmatrix}\rb.
\end{equation}
We call an $\hbar$-connection obtained in this way a Schr\"odinger-type connection. We say $Q_\cM:=Q_0$ is the underlying quadratic differential of $\cM$. We call the equation of the flat sections of this connection the Schr\"odinger equation of $\cM$.

\subsection{WKB solutions}
For a Schr\"odinger operator $\cQ$, consider a formal solution of the form $\psi(z, \hbar)=\exp{\lb\int^z \sum_{m=-1}^\infty\hbar^mP_m dz\rb}$, then one can derive a system of equations governing $P_m(z)$'s:
\begin{equation}\label{EquationP}
    \begin{split}
        &P_{-1}^2=Q_0\\
        &2P_{-1}P_m+\sum_{\substack{m_1+m_2=m-1\\ m_1, m_2\geq 0}}P_{m_1}P_{m_2}+\frac{dP_{m-1}}{dz}=Q_{m-1}(z).
    \end{split}
\end{equation}
The first equation has two solutions: $P_{-1}^{\pm}=\pm \sqrt{Q_0}$. The second equation has unique solutions for a choice of $\pm$. We write these two solutions by
\begin{equation}
\begin{split}
        P^{\pm}(z, \hbar)&:=P^{\pm}_{od}(z, \hbar)+P_{ev}(z, \hbar)\\
        P^{\pm}_{od}&:=\pm \frac{1}{2}(P^+-P^-)\\
        P_{ev}&:=\frac{1}{2}(P^++P^-)\\
\end{split}
\end{equation}
Note that $P_{ev}$ does not depend on the choice of $\pm$, and $P_{od}^+$ differs from $P_{od}^-$ only by the overall $\pm$. We also have
\begin{equation}
    P_{ev}=-\frac{1}{2}\frac{1}{P_{od}}\frac{dP_{od}}{dz}.
\end{equation}
This means $P_{ev}=\frac{\partial}{\partial z}\log(1/ \sqrt{P_{od}})$ formally.
Hence we can formally rewrite $\psi(z, \hbar)$ as
\begin{equation}\label{WKB}
    \begin{split}
        \psi_{\pm}(z, \hbar)&=\exp\lb\pm \int^zP_{od}dz\rb\exp \lb\int^zP_{ev}\rb\\
        &=\frac{1}{\sqrt{P_{od}}}\exp\lb\pm \int^zP_{od}\rb
    \end{split}
\end{equation}
These two solutions are called WKB solutions of (\ref{globalschrodinger}). 

One can expand the WKB solutions as follows:
\begin{equation}
    \psi_\pm(z, \hbar)=\exp\lb\int_{z_0}^z\frac{\pm\sqrt{Q_0}}{\hbar}dz\rb \sum^\infty_{m=0}\hbar^{m+\frac{1}{2}}\psi_m^\pm(z).
\end{equation}
The aim of exact WKB analysis is to lift these formal solutions to analytic solutions. 

\begin{definition}[Resummation]
For a point $z\in C$ and $\hbarr\in \bC^\times$, a resummation $\Psi_{\pm}(z, \hbar)$ of $\psi_{\pm}(z, \hbar)$ is defined as follows: We set
\begin{equation}
    S_\hbarr:=\lc \hbar\in \bC^\times\relmid |\hbar/\hbarr|<2, -\pi/2<\arg(\hbar/\hbarr)<\pi/2 \rc.
\end{equation}
A resummation of $\Psi_{\pm}$ at $(z, \hbarr)$ is a holomorphic function on $U\times S_\hbarr$ for some open neighborhood $U$ of $z$ such that $\Psi_{\pm}(\hbar, z)$ is asymptotically expanded to $\psi_\pm(z, \hbar)$ as $\hbar\rightarrow 0$.

\end{definition}
The rest of this section is devoted to explain the known results of resummations of WKB solutions.

\subsection{Stokes geometry}
Let $\cM$ be a meromorphic flat $\hbar$-connection of Sch\"odinger-type.
We introduce the languages of Stokes geometry in this setting. See \cite{IwakiNakanishi, BridgelandSmith} for the details. We set $L:=\SS(\cM)$.
\begin{definition}
We fix $\hbarr\in \bC^\times$.
\begin{enumerate}
    \item A turning point of $\cM$ is a branched point of $L$. We denote the set of turning points of $\cM$ by $\Zero(\cM)$. Equivalently, $\Zero(\cM)$ is the zero set of $Q_\cM$.
    \item For a turning point $v$ of $\cM$, an $\hbarr$-Stokes curve emanating from $v$ is a subset of the closure of
    \begin{equation}
        \lc z\in C\bs M\relmid \Im\lb \int_v^z\lb \frac{\xi_+(z')-\xi_-(z')}{\hbarr}\rb dz'\rb=0\rc
    \end{equation}
    such that the subset is homeomorphic to a connected interval, one of the boundary is $v$, and the other boundary is another turning point or in $M$. Here $\xi_+$ and $\xi_-$ are the restrictions of $\lambda$ to the sheets of $L$. Equivalently, $\xi_\pm=\pm\sqrt{Q_0}$.
   \item The $\hbarr$-Stokes graph $G_\cM$ of $\cM$ is the union of the $\hbarr$-Stokes curves of $\cM$.
   \item An $\hbarr$-Stokes region is a connected component of the complement of the closure of the $\hbarr$-Stokes graph in $C$.
   \item A subset of the $\hbarr$-Stokes graph is called an $\hbarr$-Stokes segment if it is diffeomorphic to a closed interval $[a,b]\subset \bR$ ($a<b$) and it connects two turning points.
   \end{enumerate}
\end{definition}

We would like to introduce an easy class of connections.
\begin{definition}[\cite{BridgelandSmith}]\label{weaklyGMN}
We say a quadratic differential $Q_0$ is (complete) GMN if 
\begin{enumerate}
 \item the order of any pole of $Q_0$ is more than or equal to 2,
    \item $Q_0$ has at least one pole and one zero,
    \item every zero of $Q_0$ is simple.
\end{enumerate}
We say a meromorphic flat $\hbar$-connection $\cM$ is GMN if the underlying quadratic differential $Q_\cM$ is GMN.
\end{definition}

For GMN connections, their Stokes regions are easy:
\begin{theorem}[{Strebel~\cite{Strebel}, Bridgeland--Smith~\cite{BridgelandSmith}}]
Suppose $\cM$ is GMN and the $\hbarr$-Stokes graph of $Q_\cM$ does not have any $\hbarr$-Stokes segments. Then an $\hbarr$-Stokes region of $\cM$ has one of the following forms:
\begin{enumerate}
    \item(horizontal strip). a square surrounded by four edges of the $\hbarr$-Stokes graph. Two of four vertices are poles of $Q_\cM$, and the others are turning points.
    \item(half plane). a region surrounded by two edges of the $\hbarr$-Stokes graph. One of the vertices is a turning point, and the others are poles of $Q_\cM$.
\end{enumerate}
\end{theorem}

\begin{assumption}
In the rest of this section, we assume that the quadratic differential $Q_0$ is GMN.
\end{assumption}

\subsection{WKB regularity}
To obtain the resummability of WKB solutions, we have to assume that the higher degree terms with respect to $\hbar$ in the differential operator is tame in some sense.

Locally around a pole $p$ of $\cM$, the equation can be written as 
\begin{equation}\label{usualschrodinger}
    ((\hbar\partial)^2-Q(z, \hbar))\psi=0.
\end{equation}

\begin{assumption}[WKB-regularity]\label{WKBregular}
We say $\cM$ is WKB-regular if, around any pole $p$, $Q(z,\hbar)=Q_0(z)+\hbar Q_1(z)+\hbar^2Q_2(z)+\cdots$ satisfies the following:
\begin{enumerate}
    \item $p$ is a pole of $Q_0$. 
    \item If $Q_0$ has a pole $p$ of order $m\geq 3$, then the pole order of $Q_i (i\geq 1)$ at $p$ is less than $1+m/2$. 
    \item If $Q_0$ has a pole of order $i=2$, the following conditions hold: $Q_i$ has at most one simple pole at $p$ for all $i\geq 1$ except for $n=2$, and $Q_2(z)$ has a double pole at $p$ and satisfies $Q_2(z)=-\frac{1}{4(z-p)^2}(1+O(z-p))$ as $z\rightarrow p$.
\end{enumerate}
\end{assumption}
This is defined in \cite{IwakiNakanishi}, which attributes it to Koike--Sch\"afke. 

We would like to explain an important property of WKB-regularity.
\begin{definition}
For a singularity $z$ of a meromorphic flat connection $\cN$ (especially, of a linear differential equation) on $C$ without $\hbar$, the formal completion of $\cN$ at $0$ is, after taking a branched covering $z\mapsto z^k$, isomorphic to the form
\begin{equation}
    \bigoplus_i \cE^{f_i}\otimes \cL_i 
\end{equation}
where $i$ runs through finite indices, each $\cL_i$ is regular singular, each $f_i$ is Laurent polynomial, and $\cE^{f_i}$ is rank 1 connection defined by $d+ df_idz$. This is known as Hukuhara--Levelt--Turrittin theorem. The set of Laurent--Puiseux polynomial $\frakI_{\cN,0 }=\{g_i(z):=f_i(z^{1/k})\}_{i\in I}$ is uniquely determined from $\cN$ up to the following operations: (i) for some $g_i$, replace it with $g_i+g$ where $g$ is a regular Puiseux polynomial, and (ii) Galois transformation. We say the class $\frakI_{\cN, 0}$ modulo the above operations {\em the irregularity type of $\cN$}. The rank of the irregularity type $\frakI_{\cN, 0}$ is defined to be $\rank \cN$. 
Let $r_i$ be the ramification degree of $g_i$. Then the rank of the irregularity type is the same as $\sum_{i\in I} r_i$.
\end{definition}

For a fixed $\hbarr$, associated to (\ref{usualschrodinger}), we consider the equation
\begin{equation}\label{hbarrschrodinger}
    \lb\partial^2-\frac{Q}{\hbarr^2}\rb\psi=0.
\end{equation}
\begin{proposition}\label{constantformal}
The irregularity type of (\ref{hbarrschrodinger}) at a singularity is given by $\xi_\pm/\hbarr$ for any $\hbarr$.
\end{proposition}
\begin{proof}
In a local coordinate around a pole, $Q_0$ can be written as $\xi=(\frac{c_1}{z^{n/2}}+\frac{c_2}{z})dz$ with $n\geq 2$ for some $c_1, c_2\in \bC$~\cite{Strebel}. Recall from (\ref{WKB}) that the WKB solution can be written as
\begin{equation}
    \frac{1}{\sqrt{P_{odd}}}\exp\lb \int P_{odd}\rb.
\end{equation}
By \cite[Proposition 2.8]{IwakiNakanishi}, $\int P_{odd, i}$ for $i\geq 0$ does not have a pole. Hence $\exp\lb \int P_{odd, \geq 0}\rb$ is a (non-Laurent) formal power series.

Also, we have 
\begin{equation}
\begin{split}
        \frac{1}{\sqrt{P_{odd}}}&= \frac{1}{\sqrt{c_1z^{-n/2}+c_2z^{-1}+P_{\geq 1, odd}}}\\
        &=\frac{1}{z^{-n/4}}\frac{1}{\sqrt{c_1+c_2z^{n/2-1}+z^{n/2}P_{\geq 1, odd}}}\\
\end{split}
\end{equation}
Again, by \cite[Proposition 2.8]{IwakiNakanishi}, $z^{n/2}P_{\geq 1, odd}$ is a formal power series. Hence the WKB solution is
\begin{equation}
    \exp\lb \int (c_1z^{-n/2}+c_2z^{-1})/\hbarr\rb F
\end{equation}
where $F$ is a (non-Laurent) Puiseux series. This confirms the claim.
\end{proof}

\subsection{Resummation}
The WKB solution series are not convergent series. The standard way to lift the WKB solutions to analytic convergent solutions is the Borel resummation method as we recall in the below.

To implement the Borel resummation, we first expand (\ref{WKB}) as
\begin{equation}
    \psi_\pm(z, \hbar)=\exp\lb\int_{z_0}^z\frac{\xi_\pm}{\hbar}dz\rb \sum^\infty_{m=0}\hbar^{m+\frac{1}{2}}\psi_m^\pm(z).
\end{equation}
The formal Borel transform with respect to $\hbar^{-1}$ is
\begin{equation}
    \psi_{\pm, B}(z, y)=\sum_{m=0}^\infty \frac{\psi^\pm_m(z)}{\Gamma(m+\frac{1}{2})}(y\pm a(z))^{m+\frac{1}{2}}
\end{equation}where $a(z):=\int^z_{z_0}\sqrt{Q}dz$. This series is convergent. We set
\begin{equation}
    \cS[\psi_\pm](z,\hbar):=\int_{\mp a(z)}^\infty e^{-y/\hbar}\psi_{\pm, B}(z, y)dy.
\end{equation}

\begin{theorem}[{Koike--Sch\"afke, \cite[Theorem 2.18]{IwakiNakanishi}}]\label{Koike}
Suppose $\cM$ is GMN, WKB-regular, and of Schr\"odinger-type. Take $\theta\in \bR/2\pi\bZ$. Suppose that the $e^{i\theta}$-Stokes graph of $\cM$ does not have any $e^{i\theta}$-Stokes segments.
Fix $z\in C\bs G_\cM$. Then there exists $\hbarr\in \bC^\times_\hbar$ such that $\arg\hbarr=\theta$, $\psi_{\pm, B}(z, y)$ is convergent in a Stokes region $D$ containing $z$, and $S[\psi_\pm]$ is a resummation  of $\psi_{\pm}$ in $D\times \cS_\hbarr$.
\end{theorem}

\begin{remark}
The sketch of the proof of the above theorem is available in Takei~\cite{Takei}, which is a generalization of the argument in \cite{DLS}.
\end{remark}

\subsection{Voros' formula}
We would like to describe a connection formula known as Voros' formula. Let $\cM$ be a meromorphic flat $\hbar$-connection of Schr\"odinger-type. We denote the underlying spin structure by $\sqrt{K}$.

We denote the spectral curve of the connection by $L$. Considering $Q_0^{1/4}$, as a section of $\sqrt{K}$, we have a (topological) branched covering $\widetilde{\pi}\colon \widetilde{L}\rightarrow L$ branched at the branch points of $\pi\colon L\rightarrow C$ as the Riemann surface of $\sqrt{K}$.
Since (\ref{WKB}) involves $Q_0^{1/4}$, to fix a branch of (\ref{WKB}), it is necessary to designate a sheet of $\widetilde{L}$. 

Now we would like to introduce Voros' connection formula. Let $D_1, D_2$ be adjacent Stokes regions. Let $z_0$ be a turning point on the Stokes curve separating $D_1$ and $D_2$. We will use the following normalization in the expression (\ref{WKB}). Fix $z\in D_1$ and $(z, \zeta)\in \pi^{-1}(z)\subset L$. Let us denote the corresponding branch of $P_{od}$ by $P_{od}^{(z, \zeta)}$. Let $\gamma_z$ be a minimal counterclockwise loop starts from $(z,\zeta)$ surrounding $z_0$ on $L$. Then we set
\begin{equation}\label{normalization}
    \int_{z_0}^{(z,\zeta)}P_{od} dz:=\frac{1}{2}\int_{{\gamma_{z}}} P_{od}^{(z,\zeta)} dz.
\end{equation}
This is called the turning point normalization at $z_0$.
We also fix a branch of $\sqrt{P_{od}}$, which is given by a choice of $\tilde{z}\in \tilde{\pi}^{-1}(z,\zeta)$. 
We will write the corresponding resummed solutions on $D_1$ by $\Psi_{z_0, \tilde{z},1}^{D_1}$ i.e.,
\begin{equation}
    \Psi_{z_0, \tilde{z},1}^{D_1}:=\frac{1}{(\sqrt{P_{od}})_{\tilde{z}}}\exp\lb\int_{z_0}^{(z,\zeta)}P_{od}dz \rb.
\end{equation}

We also take a lift of $\gamma_z$ to $\widetilde{L}$ starting from $\tilde{z}$ and denote the endpoint by $\widetilde{z}'$. We set 
\begin{equation}
    \Psi_{z_0, \tilde{z},2}^{D_1}:=\frac{1}{(\sqrt{P_{od}})_{\tilde{z}'}}\exp\lb\int_{z_0}^{\tilde{\pi}(\tilde{z}')}P_{od}dz \rb.
\end{equation}
This is also understood to be resummed. We set $\Psi_{z_0. \tilde{z}}^{D_1}:=(\Psi_{z_0, \tilde{z},1}^{D_1},\Psi_{z_0, \tilde{z},2}^{D_1})$.

Let $\gamma_{12}$ be an interval which starts from $z$, ends at $z_2\in D_2$, and intersects with the Stokes graph exactly once at a point on the Stokes curve separating $D_1$ and $D_2$. Take the lift $\widetilde{\gamma_{12}}$ of $\gamma_{12}$ to $\widetilde{L}$ starting from $\widetilde{z}$. We denote the end point of $\widetilde{\gamma_{12}}$ by $\widetilde{z_2}$. Then we define $\Psi^{D_2}_{z_0, \widetilde{z_2}}$ in the same way as $\Psi^{D_1}_{z_0, \widetilde{z_2}}$ replacing $\widetilde{z}$ with $\widetilde{z_2}$.

\begin{theorem}[{Voros~\cite{Voros}, Aoki--Kawai--Takei~\cite{AKT}}]\label{Voros}
The analytic continuation of $ \Psi_{z_0, \tilde{z}}^{D_1}$ crossing the Stokes curve is related to $\Psi_{z_0, \tilde{z}_2}^{D_2}$ by
\begin{equation}
    \Psi_{z_0, \tilde{z}}^{D_1}= \Psi_{z_0, \tilde{z}_2}^{D_2}\cdot S
\end{equation}
where
\begin{equation}
    S=\begin{cases}
    \left(
    \begin{array}{cc}
      1 & 0  \\
     -1 & 1 
    \end{array}
  \right) \text{ if $\int_v^z\frac{\sqrt{Q_0}}{\hbarr}dz>0$ on the Stokes curve, }
  \\
  \\
  \left(
    \begin{array}{cc}
      1 & -1  \\
      0 & 1 
    \end{array}
  \right)\text{ otherwise.}
    \end{cases}
\end{equation}
\end{theorem}

On a horizontal strip $D$, we have two turning points in its closure. We denote them by $v_1, v_2$. Take a path $\gamma_{v_1v_2}$ in $D$ connects $v_1$ and $v_2$. Take a lift $\widetilde{\gamma_{v_1v_2}}$ of $\gamma_{v_1v_2}$ on $\widetilde{L}$. For $\tilde{z}\in \widetilde{\gamma}$, we have two solutions normalized at $v_1$ and $v_2$ and denote them by $\Psi^D_{v_1, \tilde{z}}$ and $\Psi^D_{v_2, \tilde{z}}$. Then we have $\Psi^D_{v_1, \tilde{z}}=\exp(\int_{\tilde{\pi}(\widetilde{\gamma_{v_1v_2}})} P_{od}dz) \Psi^D_{v_2, \tilde{z}}$.

\section{Sheaf quantizations associated to meromorphic flat $\hbar$-connection of rank 2 of Schr\"odinger type}
In this section, $\cM$ is a meromorphic flat $\hbar$-connection satisfying the following assumption:
\begin{assumption}\label{ourassumption}
$\cM$ is GMN, WKB-regular, and of Scr\"odinger-type.
\end{assumption}
We denote the underlying Schr\"odinger operator by $\cQ$, and the spin structure by $\sqrt{K}$. We also set $\bK=\bC$ in this section.

\subsection{Construction}
We will first construct our sheaf quantization on a Stoke region $D$. Then $\widetilde{L}|_D\rightarrow D$ is a trivial covering of degree 4. 

For each Stokes edge $e$, we take a thickening of it $D_e$. Namely, $D_e$ is a small contractible open neighborhood of the relative interior of $e$ in $C\bs (M\cup \mathrm{Zero}(\cM))$. 

Fix a sheet $i$. We set
\begin{equation}
    s^v_{L}|_{D,i}:=\bigoplus_\pm \bK_{D^v_\pm}
\end{equation}
where
\begin{equation}
    D^v_\pm:=\lc (z, t)\in D'\times \bR_t \relmid t\geq \mp \Re\int_v^z(\sqrt{Q_0}/\hbarr) dz \rc,
\end{equation}
$v$ is a turning point in the closure of $D$, and $D'$ is an open subset which is a union of $D$ and the thickenings of the Stokes edges of $D$. This object does not depend on $i$, but we will keep the indices. We set
\begin{equation}\label{eq:standard}
    S^v_{L}|_{D,i}:=\bigoplus_{c\in \bR} T_cs^v_{L}|_{D,i}.
\end{equation}
This defines an object of $\Sh^\bR_{\tau>0}(D'\times \bR_t)$ with an obvious equivariant structure. 

We would like to list up the gluing-up isomorphisms:
\begin{enumerate}
\item (Change of sheets) For a sheet $i$, we denote the sheet where we will arrive by one counter clockwise loop around $v$ by $i'$, Then we define
\begin{equation}
    \varphi_{ii'}\colon  S^v_{L}|_{D,i}\rightarrow  S^v_{L}|_{D,i'}
\end{equation}
by $T$: namely, the multiplication by $\sqrt{-1}$. 
\item (Change of normalizations) For $v_1, v_2\in D$, the isomorphism $\varphi_{v_1v_2}\colon S^{v_1}_{L}|_{D,i}\rightarrow S^{v_2}_{L}|_{D,i}$ is induced by the scalar multiplication
\begin{equation}\label{eq:changeofnormal}
        \exp\lb \int_{\gamma_{v_1,v_2}}{P_{od}}\rb \cdot (-) \colon
        T_c \bK_{D^{v_1}_\pm}\rightarrow T_{c\mp\Re\int_{\gamma_{v_1v_2}}(\sqrt{Q_0}/\shbarr)dz} \bK_{D_{\pm}^{v_2}}
\end{equation}
where $\int_{\gamma_{v_1,v_2}}{P_{od}}$ is understood as its resummation.
\item (Change of regions) For adjacent Stokes regions $D_1$ and $D_2$, the Stokes edge $e$ separating two regions satisfies $\Re\int^x_v\sqrt{Q_0}/\hbarr dz <0$ or $\Re\int^x_v\sqrt{Q_0}/\hbarr dz >0$. 

We set
\begin{equation}
    (D_e)_\pm^v:=\lc (z, t)\in D_e\times \bR_t \relmid t\geq \mp \Re\int_v^z(\sqrt{Q_0}/\hbarr) dz \rc,
\end{equation}
which is a subset of $(D_1)_\pm^v\cap (D_2)_\pm^v$.

Then we have
\begin{equation}
\End\lb \bigoplus_\pm \bK_{(D_e)_\pm^v}\rb\\
        =\bK\cdot\id_{+}\oplus \bK\cdot \id_{-}\oplus \bK\cdot e
\end{equation}
where $\id_\pm\in \End\lb\bK_{(D_e)_\pm^v}\rb$ are the identities of the components and 
\begin{equation}
    e\in \begin{cases}
        &\Hom\lb \bK_{(D_e)_+^v}, \bK_{(D_e)_-^v}\rb \text{ if $\Re\int^x_v\sqrt{Q}/\hbarr dz >0$} \\
        & \Hom\lb \bK_{(D_e)_-^v}, \bK_{(D_e)_+^v}\rb \text{ if $\Re\int^x_v\sqrt{Q}/\hbarr dz <0$}
    \end{cases}
\end{equation}
is a canonical nontrivial basis coming from the inclusion of the closed sets. We can reflect the Voros connection formula using this:
\begin{equation}
    \varphi_{D_1, D_2}:=\id_++\id_--e\in \End\lb \bigoplus_\pm \bK_{(D'_1\cap D'_2)_\pm^v}\rb.
\end{equation}
We consider this morphism as an isomorphism $\varphi\colon S^v_{L}|_{D_1, i}|_{D_1\cap D_2}\rightarrow S^v_{L}|_{D_2, i}|_{D_1\cap D_2}$. 
\end{enumerate}

\begin{lemma}
The sheaves $S^v_{L}|_{D,i}$'s can be glued up by $\varphi$'s and give an object in $\Sh_{\tau>0}^\bR(( C\bs (M\cup \Zero(\cM)))\times \bR_t)$. We denote this object by $S^{\shbarr}_{\cM}|_{C\bs (M\cup \Zero(\cM))}$.
\end{lemma}
\begin{proof}
This is clear from the consistency of the connection formulas.
\end{proof}

We would like to extend the object to $C\bs M$. Around a turning point, we have three Stokes regions $D_1, D_2, D_3$. Hence $s^v_{L}|_{D_j,i}$ for $j=1,2,3$ will glue up to a sheaf $s^v_{L}$ and $S^v_{L}|_{D_1\cup D_2\cup D_3}=\bigoplus_{c\in \bR} T_cs^v_L$. Note that there are no monodromies of the solutions, hence we have 
\begin{equation}
    \Ext^1\lb \bigoplus_{c\in \bR}\bK_v\boxtimes \bK_{[c, \infty)}^{\oplus 2}, \bigoplus_{c\in \bR}T_c s^v_{L}\rb=\Lambda_0^{\oplus 2}.
\end{equation}
where $\bK_v$ is the skyscraper sheaf. By using $1\oplus 1$ in this extension space for each turning point, we get an object on $(C\bs M)\times \bR_t$. By the zero extension, we get an object $S^{\shbarr}_{\cM}\in \Sh^\bR_{\tau>0}(C\times \bR_t)$. As a summary, we have the following:

\begin{theorem}
We have a simple sheaf quantization $S^\shbarr_{\cM}\in \Sh_{\tau>0}^\bR(C\times \bR_t)$ of $L$ at $\hbarr$.
\end{theorem}
\begin{proof}
    Outside the turning points, the microsupport estimate follows from Lemma~\ref{lem:microsuppestimate1} and, the object is simple by the computation from Example~\ref{ex:graphmicrolocalization}. Around the turning points, by construction, the sheaf takes the form of $\bigoplus_{c\in \bR}T_c\cE$ where $\cE$ is the exact sheaf quantization of the local model of the vertex in \cite[Figure in p.8, and \S 4.1.2]{treumann2017cubic}, hence the desired microsupport estimate and the simpleness hold also around the turning points. This completes the proof.
\end{proof}

\subsection{Voros symbol}
We would like to recall the notion of Voros symbol. We refer to \cite{DDP} and \cite{IwakiNakanishi} for more accounts.

The ``dual" of a segments-free $\hbarr$-Stokes graph gives an ideal triangulation of $C$ with vertices in $M$. In this situation, $H_1(L, \bZ)$ is spanned by the following~\cite{BridgelandSmith}:
\begin{enumerate}
    \item The pull back of cycles in $C\bs M$.
    \item In an $\hbarr$-Stokes region of horizontal strip type, take a path connecting two turning points. This is lifted to a cycle in the spectral curve. We call this cycle, a {\em Voros edge cycle}.
\end{enumerate}
\begin{definition}
For a cycle $\gamma\in H_1(L, \bZ)$, the Voros symbol is defined by
\begin{equation}
    V_{\gamma}(\hbarr):=\oint_\gamma P_{od}(\gamma, \hbarr). 
\end{equation}
The local system on $L$ whose monodromy along $\gamma$ given by $e^{V_\gamma(\shbarr)}$ is called the $\hbarr$-Voros local system.
\end{definition}
In \cite{IwakiNakanishi}, it was proved that the collection $\{e^{V_\gamma(\shbarr)}\}$ is a set of cluster variables. Regarding this aspect, we sometimes call the set $\{e^{V_\gamma(\shbarr)}\}$ (or the Voros local system) {\em Voros--Iwaki--Nakanishi coordinate}.

Note that there is another class of Voros symbols associated to {\em paths}. We do not recall it here, since we do not use it in this paper.

\subsection{Microlocalization}
We would like to equip $L$ with a Maslov grading and a relative Pin structure. 

We will use the standard Maslov map $\alpha$ on $T^*C$~\cite{NZ}. Then the zero section is canonically equipped with a Maslov grading $\tilde{\alpha}=0$. For any Riemann surface, the second Stiefel--Whitney class always vanishes. Hence the notion of relative Pin structures and Spin structures coincide.

Let $B$ be the set of branched points in $L$.
Since we have a trivial grading on the zero section, we can induce a grading on $L\bs B$. In other words, we have a lift of $L\bs \pi^{-1}(M)\rightarrow U(1)$ to $L\bs B\rightarrow \bR$. The obstruction class $c\in H^1(L, L \bs B ,\bZ)$ to extend this lift to $L$ is zero, since the grading can be extended to $B=\Zero(\cM)$ on the base space. Hence we have a unique extension.

For a spin structure, we use $\sqrt{K}|_{C\bs M}$ to have a spin structure on $L\bs B$. We have a diagram
\[
\xymatrix{
L\bs B \ar[d] \ar[r] & B\mathrm{Spin}(2) \ar[d]  \\
L \ar[r] & B\mathrm{SO}(2).
}
\]
The obstruction to extend the spin structure lives in $H^2(L, L\bs B, \bZ/2\bZ)$. Again, the obstruction vanishes. The classes of the extensions is a torsor over $H^1(L, L\bs B, \bZ/2\bZ)$. We will choose one in the proof of the following corollary.

\begin{corollary}
The $0$-th cohomology of the microlocalization of $S^\hbarr_{\cM}$ along a brane structure is the Voros local system.
\end{corollary}
\begin{proof}
We explain with Picture~\ref{Diagram}. In this picture, $v, w$ are turning points. Lines emanating from turning points are Stokes lines. 

The inverse image of the dotted line connecting $v$ and $w$ along the projection $\pi_L\colon L\rightarrow C$ is a Voros edge cycle. We would like to compute the monodromy of the microlocalization of $\cS^\hbarr_\cM$ along the Voros edge cycle.

Take a circle $\gamma'$ encircling the dotted line. Take a connected component of $\pi_L^{-1}(\gamma)$, and name it $\gamma$. The $\gamma$ is projected down to the circle in Figure~\ref{Edge}.
We denote the sheet of $L|_{D_3}$ above the dotted line over where $\gamma$ is lying by $L_{3+}$. Below the dotted line, $\gamma$ is lying over the other sheet $L_{3-}$. Running around the tuning points change the sheets. 

Since $\gamma$ is homotopy equivalent to the Voros edge cycle, we compute the monodoromy along $\gamma$.
we can reduce the monodromy computation
\begin{figure}[htbp]
\begin{center}
\includegraphics[width=7.0cm]{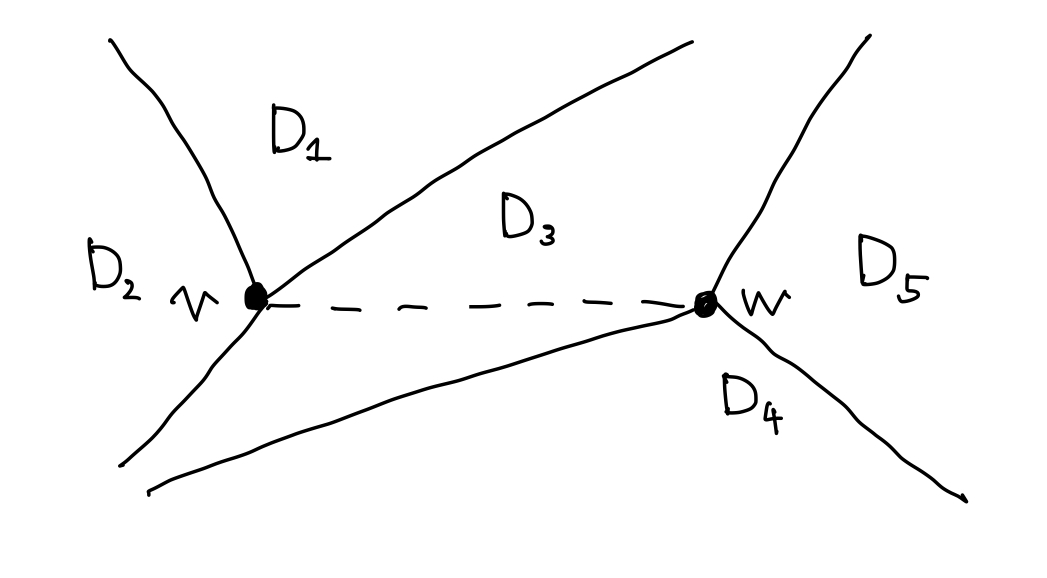}
\end{center}
\caption{Stokes diagram}\label{Diagram}
\end{figure}
\begin{figure}[htbp]
\begin{center}
\includegraphics[width=7.0cm]{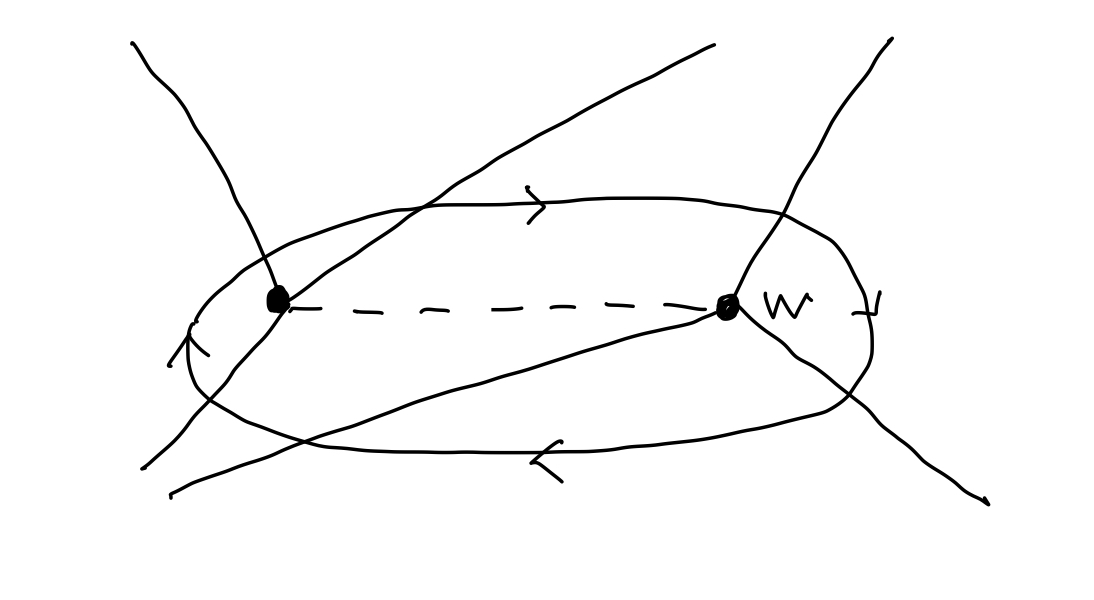}
\end{center}
\caption{Edge cycle}\label{Edge}
\end{figure}
Let us first consider over $D_3$ with the normalization $v$. Namely, we consider the trivialization $\cS^\hbarr_\cM|_{D_3}\cong \bigoplus_{c\in \bR} T_cs^v_{L}|_{D,i}$. Then we get the rank 1 constant sheaf $\bK_{L_{3+}}$ on $L_{3+}$ by the microlocalization as in Example~\ref{ex:graphmicrolocalization}. Next, we consider over $D_3$ with the normalization $w$. We similarly get $\bK_{L_{3+}}$ as a result of the microlocalization. By (\ref{eq:changeofnormal}), these two constant sheaves are glued by the multiplication of $\exp\lb \int_{\gamma_{vw}}{P_{od}}\rb$. 

Similarly, over $D_5$ and $D_4$, we get the rank 1 constant sheaves $\bK_{L, 4+}$ and $\bK_{L_5+}$ under the microlocalization. From the Voros connection formula, they are glued with $\bK_{L_{3+}}$ possibly multiplied with $\pm 1$.

Returning to $v$, since we are now going back over the other sheet, the sheaves are glued by the multiplication of $\exp\lb \int_{\gamma_{wv}}{P_{od}}\rb$. Again, tuning around the turning point possibly multiply $\pm 1$. 

Hence we get the Voros monodromy up to sign. One can fix the sign by retaking the spin structure. 
\end{proof}

\begin{remark}
A Stokes saddle is a Stokes line ends in a turning point. In our construction, we choose a generic $\hbarr$ which does not have any Stokes saddles. If one varies $\hbarr$, then the microlocalization of the sheaf quantization varies smoothly unless $\hbarr$ does not meet a point with Stokes saddles. If $\hbarr$ crosses a point where it has a Stokes saddle, the microlocalization jumps following a cluster transformation. This is an interpretation of the result by Iwaki--Nakanishi~\cite{IwakiNakanishi}.
\end{remark}

\section{Irregular Riemann--Hilbert correspondence}
In this section, we would like to relate our sheaf quantization to some known formalisms of irregular Riemann--Hilbert correspondence. In the following, we fix a sufficiently small $\hbarr\in \bC^\times$. In this section, it is convenient to think $S^\shbarr_{\cM}$ as a sheaf quantization of $L/\hbarr$ at $\hbarr=1$ with respect to the standard symplectic structure. 

\subsection{D'Agnolo--Kashiwara's formalism}
We would like to relate our sheaf quantization to D'Agnolo--Kashiwara's holonomic Riemann--Hilbert correspondence~\cite{DK}. For the details, we refer to \cite{DK, KSregularandirregular}. 

We briefly recall the statement of \cite{DK}. For a complex manifold $X$, let us denote the derived category of holonomic $\cD$-modules by $D^b_{hol}(\cD_X)$. 

For the topological side, we introduce the notion of enhanced ind-sheaves. Let $\overline{\bR}$ be the two point compactification of $\bR$ i.e., $\bR\cong (0,1)\hookrightarrow [0,1]=\overline{\bR}$. The category of enhanced ind-sheaves is defined in two steps: First, we set 
\begin{equation}
D^b(\mathrm{I}\bC_{X\times(\overline{\bR},\bR)}):=D^b(\mathrm{I}\bC_{X\times\overline{\bR}})/D^b(\mathrm{I}\bC_{X\times{\overline{\bR}\bs \bR}})
\end{equation}
where $D^b(\mathrm{I}\bC_X)$ is the bounded derived category of ind-sheaves over $X$. 
We set $\bC_{t\lesseqgtr 0}:=\bC_{\lc (x,t)\in X\times \overline{\bR}\relmid t\in \bR, t\lesseqgtr 0 \rc}$. We denote the convolution product along $\bR_t$ by $\potimes$. We set 
\begin{equation}
\mathrm{IC}_{t^*=0}:=\lc K\relmid K\potimes \bC_{\leq 0}\simeq 0, K\potimes \bC_{\geq 0}\simeq 0\rc.
\end{equation}
The category of enhanced ind-sheaves over $X$ is defined by 
\begin{equation}
E^b({\mathrm{I}\bC_X}):=D^b(\mathrm{I}\bC_{X\times (\overline{\bR}, \bR)})/\mathrm{IC}_{t^*=0}.
\end{equation}
We set
\begin{equation}
\bC^E_X:=\underset{a \rightarrow\infty}{\forlim}\bC_{t\geq a}
\end{equation}
as an object of $E^b({\mathrm{I}\bC_X})$. As usual, $\forlim$ means Ind-colimit. An object of $E^b({\mathrm{I}\bC_X})$ is said to be $\bR$-constructible if there locally exists an $\bR$-constructible sheaf $\cE$ and the object is isomorphic to $\cE\potimes \bC^E_X$. 
\begin{theorem}[\cite{DK}]
There exists a contravariant fully faithful functor $\mathrm{Sol}^E\colon D^b_{hol}(\cD_X)\hookrightarrow E^b(\mathrm{I}\bC_X)$.
\end{theorem}

Let us construct enhanced ind-sheaves by modifying our construction of sheaf quantization.

We recall our construction of sheaf quantization: It locally consists of 
$S^v_{L}|_{D,i}=\bigoplus_{c\in \bR} T_cs^v_{L}|_{D,i}$. They are glued together by the morphisms appeared in \S 8.1. Each of the gluing morphisms takes the form
\begin{equation}
    \bigoplus_{c\in \bR} T_c(\bK_{A}\oplus \bK_B)\rightarrow \bigoplus_{c\in \bR} T_c(\bK_{A}\oplus \bK_B)
\end{equation}
for some constant sheaves $\bK_A, \bK_B$. By the definition of the gluing morphisms, each of them is the direct sum of the translations ($=\bigoplus_{c\in \bR}T_c$) of a morphism of the form
\begin{equation}
     \bK_A\oplus \bK_B\rightarrow T_c\bK_A\oplus T_{c'}\bK_B.
\end{equation}
for some $c,c'\in \bR$. Applying $\potimes \bC^E_M$, we obtain an isomorphism
\begin{equation}
    (\bK_A\oplus \bK_B)\potimes \bC^E_M\rightarrow (T_c\bK_A\oplus T_{c'}\bK_B)\potimes \bC^E_M\cong (\bK_A\oplus \bK_B)\potimes \bC^E_M.
\end{equation}
Hence, these morphisms give the gluing morphisms of $s^v_{L}|_{D,i}\potimes \bC^E_M$. We denote the resulting enhanced ind-sheaf by $E(S^\hbarr_\cM)$. 
\begin{proposition}
We have $\mathrm{Sol}^E(\cM^\hbarr)\cong E(S^\hbarr_{\cM})$.
\end{proposition}
\begin{proof}
Each gluing morphism $\bK_A\oplus \bK_B\rightarrow T_c\bK_A\oplus T_{c'}\bK_B$ is given by the connection matrix between the bases of solutions given by resummed WKB solutions with certain normalizations. Here the normalizations are specified in \S 8.1 in terms of turning points denoted by $v,v_1, v_2$. Also, around the poles, the connection matrices are precisely Stokes matrices, since Proposition~\ref{constantformal} tells us that the resummed WKB solutions form a basis compatible with the Stokes filtration.

Hence the induced isomorphisms $(\bK_A\oplus \bK_B)\potimes \bC^E_M\rightarrow (\bK_A\oplus \bK_B)\potimes \bC^E_M$ constructed above is given by connection/Stokes matrices of the given flat connection $\cM^\hbarr$. 

On the other hand, \cite[\S 9.8]{DK} tells us that this gluing of enhanced ind-sheaves by connection/Stokes matrices is precisely their solution sheaf. This completes the proof.
\end{proof}

\begin{remark}
    If one restricts $\Sol^E(\cM^\hbarr)$ to the complement of the divisor, we obtain the usual solution local system of the given differential equation. In this way, we can recover the usual solution local system from $S^\hbarr_\cM$.
\end{remark}

\begin{remark}
It is desirable to make our construction here functorial.
\end{remark}

\subsection{Irregularity knots and non-Novikovized quantization}

We next compare our sheaf quantization with the Legendrian knot description of Stokes data~\cite{STWZ}. In this section, we set $L_{\cM}:=L$.

Let us fix a divisor $M$ on $C$. Fix a positive integer $k$. For each point $p\in M$, we fix an irregularity type of rank $k$. We denote this data $F_M$. Consider meromorphic connections with the given irregularity type $F_M$. 

We first define the notion of irregularity knot. For $p\in M$, we denote the irregularity type of $p$ by $\{g_1,..., g_k\}$. Take a small disk $D_p$ around the singularity and fix a coordinate $z$ on the disk. For each $g_i$, we give an immersion of $\sqcup_k S^1$ in $D_p$ by 
\begin{equation}
    \bigcup_{i=1}^k\lc(\theta, \epsilon\cdot \Re(g_i(e^{\sqrt{-1}\theta}/r))\relmid \theta\in S^1\rc
\end{equation}
in the polar coordinate. Here $\epsilon$ is taken sufficiently small so that the immersion is into $D_p$ and $r$ is some large positive number. Putting the outward conormals on the immersion, we have a Legendrian link on the cosphere bundle $S^*D_p$. For a sufficiently large $r$, the Legendrian isotopy class of the link does not depend on $r$. We call this link, the irregularity knot and denote it by $L_{\{g_i\}}$. We also identify the Legendrian link with the associated conical Lagrangian in $T^*D_p\subset T^*C$. 

To sum up, associated to the given data of irregularity types $F_M$, we get a conical Lagrangian $L_{F_M}$ contained in $\bigcup_{p\in M}T^*D_p\subset T^*C$.

We denote the following category by $\Sh^1_{L_{F_M}}(C)$: The subcategory of constructible sheaves over $C$ consisting of the objects satisfying the following conditions:
\begin{assumption}\label{assump1}
\begin{enumerate}
    \item microsupports are contained in $L_{F_M}\cup T^*_CC$,
    \item microstalks are rank 1 and concentrated in degree 0,
    \item stalks are zero at any $p\in M$.
\end{enumerate}
\end{assumption}

\begin{theorem}[\cite{STWZ}]\label{STWZRH}
There exists an equivalence between the category of meromorphic connections with the irregularity types of $F_M$ and $\Sh^1_{L_{F_M}}(C)$. We denote the equivalence functor by $\Sol_{STWZ}$.
\end{theorem}

We relate this equivalence with our sheaf quantization by slightly modifying the construction of $S^{\hbarr}_{\cM}$ around the poles. 

Let $p$ be a pole of the differential equation and take a local coordinate $z$ around $p$. Then we consider the following curved cylinder
\begin{equation}
    \cC_p:=\lc(z,t)\relmid |z-p|\leq \eta \rc
    \cup \lc (z, t)\relmid \substack{|z-p|>\eta, 0< \epsilon(|z-p|-\eta)< \pi,\\t\geq \tan(-\pi/2+\epsilon(|z-p|-\eta))}\rc \subset C\times \bR_t
\end{equation}
where $\epsilon$ (resp. $\eta$)is sufficiently large (resp. small) positive number. On each Stokes region $D$, we take 
\begin{equation}
    D^\pm:=\pi^{-1}\pi(\lc(x, t)\in D \times \bR_t\relmid t\geq -\Re \int \frac{\xi_\pm}{\hbarr}dz\rc \bs \bigcup_{p\in \overline{D}}\cC_{p})\cap \lc(x, t)\in D\times \bR_t\relmid t\geq -\Re\int \frac{\xi_\pm}{\hbarr}dz\rc.
\end{equation}
where $\pi\colon D\times \bR\rightarrow D$ is the projection. Figure~\ref{modification} is an example of the situation where the gray shaded region is an example of $D^\pm$.
\begin{figure}[htbp]
\begin{center}
\includegraphics[width=10.0cm]{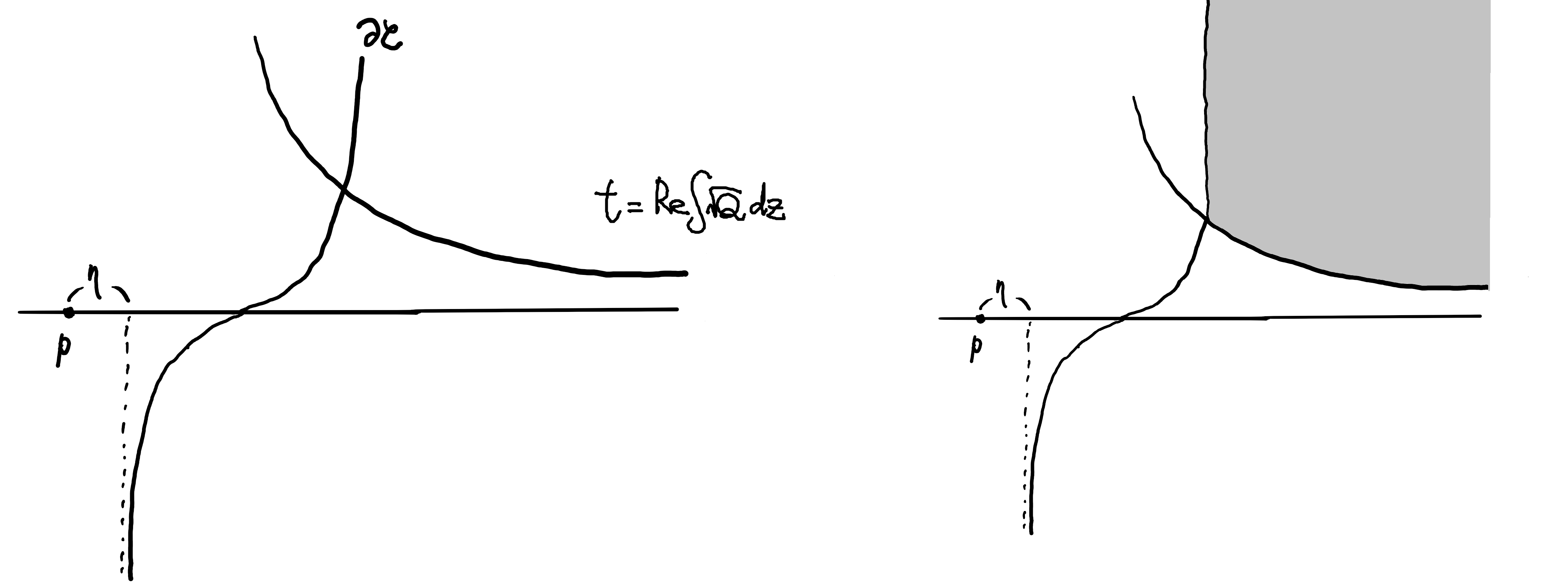}
\end{center}
\caption{An example of $D^\pm$}
\label{modification}
\end{figure}
We set
\begin{equation}
    S_{L_c}^D:=\bigoplus_{c\in \bR} T_c\lb\bigoplus_\pm \bK_{D^\pm}\rb.
\end{equation}

We can glue up this by the same rule as in the construction of $S^{\shbarr}_{\cM}$. We denote the resulting sheaf $S^{\shbarr}_{\cM, c}$ and we set $L_{\cM, c}:=\musupp(S^{\shbarr}_{\cM,c})$. Note that $L_{\cM, c}$ is conical and equal to $L_{F_M}$ outside some neighborhood of the zero section.

\begin{proposition}\label{conicVoros}
\begin{enumerate}
    \item There exists a Lagrangian homotopy between $L_\cM/\hbarr$ and $L_{\cM,c}$. 
    \item The microlocalization of $S^{\shbarr}_{\cM, c}$ is the same as the microlocalization of $S^{\shbarr}_{\cM}$ along the Lagrangian homotopy.
\end{enumerate}
\end{proposition}
\begin{proof}
\begin{enumerate}
    \item Since, the both $L_\cM$ and $L_{\cM, c}$ are Lagrangian homotopic to $L_\cM\bs T^*\pi(\cC)$ where $\cC=\bigcup_p \cC_p$.
    \item The microlocalization is determined by the microlocalization at $L_\cM\bs T^*\pi(\cC)$.
\end{enumerate}
\end{proof}

To relate $\cS^\hbarr_{\cM,c}$ to the STWZ functor, we consider the following construction. Let $\cE$ be an $\bR$-equivariant sheaf on $C\times \bR_t$. For an open subset $U$ of $C$, we set
\begin{equation}
    \cE'_\Lambda(U):=\Hom_{\Sh^\bR_{\tau >0}(U\times \bR_t)}(\bigoplus_{c\in \bR}\bK_{U\times [c, \infty)},\cE|_U)\otimes_{\Lambda_0}\Lambda.
\end{equation}
This assignment defines a presheaf on $C$ valued in $\Lambda$-vector spaces. We denote the sheafification of $\cE'_\Lambda$ by $\cE_\Lambda$.

\begin{proposition}\label{projection}
The sheaf $(S^\shbarr_{\cM, c})_{\Lambda}$ is isomorphic to $\Sol_{STWZ}(\cM^\hbarr)\otimes_\bC\Lambda$.
\end{proposition}
\begin{proof}
Note that $\pi(\{t=\Re\int\frac{\xi_\pm}{\shbarr}\}\cap \partial \cC)$ coincides with the projection of the irregularity knot. The gluing maps are responsible for Stokes phenomena. 
\end{proof}
\footnotesize
\bibliographystyle{alpha}
\bibliography{bibs.bib}

\newcommand{\etalchar}[1]{$^{#1}$}
\begin{thebibliography}{STWZ19}

\bibitem[AKT91]{AKT}
Takashi Aoki, Takahiro Kawai, and Yoshitsugu Takei.
\newblock The {B}ender-{W}u analysis and the {V}oros theory.
\newblock In {\em Special functions ({O}kayama, 1990)}, ICM-90 Satell. Conf. Proc., pages 1--29. Springer, Tokyo, 1991.

\bibitem[BL94]{BernsteinLunts}
Joseph Bernstein and Valery Lunts.
\newblock {\em Equivariant sheaves and functors}, volume 1578 of {\em Lecture Notes in Mathematics}.
\newblock Springer-Verlag, Berlin, 1994.

\bibitem[BS15]{BridgelandSmith}
Tom Bridgeland and Ivan Smith.
\newblock Quadratic differentials as stability conditions.
\newblock {\em Publ. Math. Inst. Hautes \'{E}tudes Sci.}, 121:155--278, 2015.

\bibitem[DDP93]{DDP}
\'{E}. Delabaere, H.~Dillinger, and F.~Pham.
\newblock R\'{e}surgence de {V}oros et p\'{e}riodes des courbes hyperelliptiques.
\newblock {\em Ann. Inst. Fourier (Grenoble)}, 43(1):163--199, 1993.

\bibitem[DFK{\etalchar{+}}]{Mulaseetal}
Olivia Dumitrescu, Laura Fredrickson, Georgios Kydonakis, Rafe Mazzeo, Motohico Mulase, and Andrew Neitzke.
\newblock Opers versus nonabelian hodge.
\newblock {\em eprint arXiv:1607.02172}.

\bibitem[DK16]{DK}
Andrea D'Agnolo and Masaki Kashiwara.
\newblock Riemann-{H}ilbert correspondence for holonomic {D}-modules.
\newblock {\em Publ. Math. Inst. Hautes \'{E}tudes Sci.}, 123:69--197, 2016.

\bibitem[DLS93]{DLS}
T.~M. Dunster, D.~A. Lutz, and R.~Sch\"{a}fke.
\newblock Convergent {L}iouville-{G}reen expansions for second-order linear differential equations, with an application to {B}essel functions.
\newblock {\em Proc. Roy. Soc. London Ser. A}, 440(1908):37--54, 1993.

\bibitem[Gai]{GaiottoTBA}
Davide Gaiotto.
\newblock Opers and tba.
\newblock {\em eprint arXiv:1403.6137}.

\bibitem[Gro57]{Grothendieck}
Alexander Grothendieck.
\newblock Sur quelques points d'alg\`ebre homologique.
\newblock {\em Tohoku Math. J. (2)}, 9:119--221, 1957.

\bibitem[GS11]{GS}
Mark Gross and Bernd Siebert.
\newblock From real affine geometry to complex geometry.
\newblock {\em Ann. of Math. (2)}, 174(3):1301--1428, 2011.

\bibitem[GT13]{GetmanekoTamarkin}
Alexander Getmanenko and Dmitry Tamarkin.
\newblock Microlocal properties of sheaves and complex {WKB}.
\newblock {\em Ast\'{e}risque}, (356):x+111, 2013.

\bibitem[Gui15]{guillermou2015quantization}
Stéphane Guillermou.
\newblock Quantization of conic lagrangian submanifolds of cotangent bundles, 2015.

\bibitem[Gui16]{Guillermou}
St{\'e}phane Guillermou.
\newblock Quantization of exact lagrangian submanifolds in a cotangent bundle, lectures at the 2016 summer school ``symplectic topology, sheves, and mirror symmetry.
\newblock {\em available at the author's webpage}, 2016.

\bibitem[IK]{IK}
Yuichi Ike and Tatsuki Kuwagaki.
\newblock Microlocal categories over the novikov ring i: cotangent bundles.

\bibitem[IN14]{IwakiNakanishi}
Kohei Iwaki and Tomoki Nakanishi.
\newblock Exact {WKB} analysis and cluster algebras.
\newblock {\em J. Phys. A}, 47(47):474009, 98, 2014.

\bibitem[Jin]{Jin}
Xin Jin.
\newblock Microlocal sheaf categories and the j-homomorphism.
\newblock {\em arXiv:2004.14270}.

\bibitem[JT]{JT}
Xin Jin and David Treumann.
\newblock Brane structures in microlocal sheaf theory.
\newblock {\em eprint arXiv:1704.04291}.

\bibitem[Kon]{KonTalk}
Maxim Kontsevich.
\newblock Resurgence and quantization.
\newblock {\em YouTube, IH\`ES}.

\bibitem[KS]{KontsevichSoibelman}
Maxim Kontsevich and Yan Soibelman.
\newblock Holomorphic floer theory.

\bibitem[KS94]{KS}
Masaki Kashiwara and Pierre Schapira.
\newblock {\em Sheaves on manifolds}, volume 292 of {\em Grundlehren der Mathematischen Wissenschaften [Fundamental Principles of Mathematical Sciences]}.
\newblock Springer-Verlag, Berlin, 1994.
\newblock With a chapter in French by Christian Houzel, Corrected reprint of the 1990 original.

\bibitem[KS16]{KSregularandirregular}
Masaki Kashiwara and Pierre Schapira.
\newblock {\em Regular and irregular holonomic {D}-modules}, volume 433 of {\em London Mathematical Society Lecture Note Series}.
\newblock Cambridge University Press, Cambridge, 2016.

\bibitem[KT05]{KawaiTakei}
Takahiro Kawai and Yoshitsugu Takei.
\newblock {\em Algebraic analysis of singular perturbation theory}, volume 227 of {\em Translations of Mathematical Monographs}.
\newblock American Mathematical Society, Providence, RI, 2005.
\newblock Translated from the 1998 Japanese original by Goro Kato, Iwanami Series in Modern Mathematics.

\bibitem[Kuw22]{hRH}
Tatsuki Kuwagaki.
\newblock $\hbar$-{R}iemann-{H}ilbert correspondence, 2022.

\bibitem[NZ09]{NZ}
David Nadler and Eric Zaslow.
\newblock Constructible sheaves and the {F}ukaya category.
\newblock {\em J. Amer. Math. Soc.}, 22(1):233--286, 2009.

\bibitem[PS04]{PS}
Pietro Polesello and Pierre Schapira.
\newblock Stacks of quantization-deformation modules on complex symplectic manifolds.
\newblock {\em Int. Math. Res. Not.}, (49):2637--2664, 2004.

\bibitem[Sei00]{SeidelGraded}
Paul Seidel.
\newblock Graded {L}agrangian submanifolds.
\newblock {\em Bull. Soc. Math. France}, 128(1):103--149, 2000.

\bibitem[Str84]{Strebel}
Kurt Strebel.
\newblock {\em Quadratic differentials}, volume~5 of {\em Ergebnisse der Mathematik und ihrer Grenzgebiete (3) [Results in Mathematics and Related Areas (3)]}.
\newblock Springer-Verlag, Berlin, 1984.

\bibitem[STWZ19]{STWZ}
Vivek Shende, David Treumann, Harold Williams, and Eric Zaslow.
\newblock Cluster varieties from {L}egendrian knots.
\newblock {\em Duke Math. J.}, 168(15):2801--2871, 2019.

\bibitem[Tak17]{Takei}
Yoshitsugu Takei.
\newblock W{KB} analysis and {S}tokes geometry of differential equations.
\newblock In {\em Analytic, algebraic and geometric aspects of differential equations}, Trends Math., pages 263--304. Birkh\"{a}user/Springer, Cham, 2017.

\bibitem[Tam18]{Tam}
Dmitry Tamarkin.
\newblock Microlocal condition for non-displaceability.
\newblock In {\em Algebraic and analytic microlocal analysis}, volume 269 of {\em Springer Proc. Math. Stat.}, pages 99--223. Springer, Cham, 2018.

\bibitem[TZ17]{treumann2017cubic}
David Treumann and Eric Zaslow.
\newblock Cubic planar graphs and legendrian surface theory, 2017.

\bibitem[Vit]{Vit}
Claude Viterbo.
\newblock An introduction to symplectic topology through sheaf theory.
\newblock {\em http://www.math.polytechnique.fr/cmat/viterbo/Eilenberg/Eilenberg.pdf}.

\bibitem[Vor83]{Voros}
A.~Voros.
\newblock The return of the quartic oscillator: the complex {WKB} method.
\newblock {\em Ann. Inst. H. Poincar\'{e} Sect. A (N.S.)}, 39(3):211--338, 1983.

\end{thebibliography}

\noindent
Department of Mathematics, Graduate School of Science, Kyoto University, tatsuki.kuwagaki.a.gmail.com
\end{document}